\documentclass[final]{siamart}
\include{macro_list}

\usepackage{tikz,pgflibraryplotmarks}
\usepackage{pgf,pgfplots,pgfarrows} %% pdf-drawing package

\newdimen\iwidth
\newdimen\iheight

\usepackage{multirow}

\usepackage{amsmath,amssymb} %,amsthm}
\usepackage{algorithm,algorithmic}
\newcommand{\norm}[2]{\|#1 \|_{#2}}
\newcommand{\normtwo}[1]{\|#1\|_2}
\newcommand{\trace}{\text{trace}}
\newcommand{\krylov}[1]{\mathcal{K}_{#1}}
\newcommand{\mb}[1]{\mathbb{#1}}
\newcommand{\mc}[1]{\mathcal{#1}}

\newcommand{\diag}[1]{\text{diag}\{ #1\}}

\newcommand{\Span}[1]{\text{Span}\left\{#1\right\}}
\newtheorem{prop}{Proposition}

\usepackage{color}

\numberwithin{theorem}{section}
\numberwithin{algorithm}{section}

\newcommand{\B}[1]{\boldsymbol{#1}}

\newcommand{\TheTitle}{Generalized hybrid iterative methods for large-scale Bayesian inverse problems}
\newcommand{\TheAuthors}{J. Chung and A. K. Saibaba}

\headers{Generalized hybrid iterative methods}{\TheAuthors}

\title{{\TheTitle}}
\author{
  Julianne Chung\thanks{Department of Mathematics, Virginia Tech, Blacksburg, VA
    (\email{jmchung@vt.edu}, \url{http://www.math.vt.edu/people/jmchung/}).}
  \and
  Arvind K. Saibaba\thanks{Department of Mathematics, North Carolina State University, Raleigh, NC
    (\email{asaibab@ncsu.edu}, \url{http://www4.ncsu.edu/\~asaibab/}).}
}

\ifpdf
\hypersetup{
  pdftitle={\TheTitle},
  pdfauthor={\TheAuthors}
}
\fi

\begin{document}
\maketitle

\begin{abstract}
We develop a generalized hybrid iterative approach for computing solutions to large-scale Bayesian inverse problems.  We consider a hybrid algorithm based on the generalized Golub-Kahan bidiagonalization for computing Tikhonov regularized solutions to problems where explicit computation of the square root and inverse of the covariance kernel for the prior covariance matrix is not feasible.  This is useful for large-scale problems where covariance kernels are defined on irregular grids or are only available via matrix-vector multiplication, e.g., those from the Mat\'{e}rn class. We show that iterates are equivalent to LSQR iterates applied to a directly regularized Tikhonov problem, after a transformation of variables, and we provide connections to a generalized singular value decomposition filtered solution. Our approach shares many benefits of standard hybrid methods such as avoiding semi-convergence and automatically estimating the regularization parameter.  Numerical examples from image processing demonstrate the effectiveness of the described approaches.
\end{abstract}

	% REQUIRED

	% REQUIRED

\section{Introduction}
Inverse problems are prevalent in many important scientific applications, and computing solutions can be challenging, especially for large-scale problems \cite{Hansen2010,calvetti2007introduction}.  In this paper, we consider large-scale inverse problems of the form
\begin{equation}
	\label{eq:problem}
	\bfd = \bfA \bfs + \bfepsilon,
\end{equation}
where $\bfd\in \bbR^m$ contains the observed data, $\bfA \in \bbR^{m \times n}, m \geq n$\footnote{For clarity of presentation, we assume $m \geq n$, but these methods apply also to problems where $m<n.$} models the forward process, $\bfs \in \bbR^n$ represents the desired parameters, and $\bfepsilon \in \bbR^{m}$ represents noise in the data.  We assume that $\bfepsilon\sim\mathcal{N}(\bfzero,\bfR)$ where $\bfR$ is a positive definite matrix whose inverse and square root are inexpensive (e.g., a diagonal matrix with positive diagonal entries).
The goal is to compute an approximation of $\bfs,$ given $\bfd$ and $\bfA.$

We are interested in \emph{ill-posed} inverse problems, where the challenge is that small errors in the data may lead to large errors in the computed approximation of $\bfs.$  Regularization is required to stabilize the inversion process.  There are many forms of regularization.  Here, we follow a Bayesian framework, where we assume a prior for $\bfs$.  That is, we treat $\bfs$ as a Gaussian random variable with mean $\bfmu \in \bbR^n$ and covariance matrix $\bfQ$.  That is, $\bfs \sim \mathcal{N}(\bfmu,\lambda^{-2}\bfQ)$, where
$\lambda$ is a scaling parameter (yet to be determined) for the precision matrix.
Specific choices for $\bfQ$ will be discussed in Section~\ref{sec:Bayes}.  Using Bayes theorem, the posterior probability distribution function is given by
\[ p(\bfs|\bfd) \propto p(\bfd|\bfs)p(\bfs) = \exp\left(-\frac{1}{2}\|\bfA \bfs - \bfd \|_{\bfR^{-1}}^2 -  \frac{\lambda^2}{2}\| \bfs-\bfmu\|_{\bfQ^{-1}}^2\right),\]
where $\norm{\bfx}{\bfM} = \sqrt{\bfx\t \bfM \bfx}$ is a vector norm for any symmetric positive definite matrix $\bfM$. The maximum a posteriori (MAP) estimate provides a solution to~\eqref{eq:problem} and can be obtained by minimizing the negative log likelihood of the posterior probability distribution function, i.e.
\begin{align}\bfs_\lambda & = \argmin_\bfs\> -\log p(\bfs|\bfd) \nonumber \\
	& = \argmin_\bfs \> \frac{1}{2}\|\bfA \bfs -\bfd \|_{\bfR^{-1}}^2 +  \frac{\lambda^2}{2}\| \bfs-\bfmu\|_{\bfQ^{-1}}^2\,,
\label{eqn:LS_orig}
	\end{align}
which is equivalent to the solution of the following normal equations,
\begin{equation}\label{eqn:normal}
(\bfA\t \bfR^{-1}\bfA  + \lambda^2 \bfQ^{-1} ){\bfs} = \bfA\t \bfR^{-1} \bfd +\lambda^2 \bfQ^{-1}\bfmu .\end{equation}
In fact, the MAP estimate, $\bfs_\lambda$, is a Tikhonov-regularized solution, and iterative methods have been developed for computing solutions to the equivalent general-form Tikhonov problem,
\begin{equation}
	\label{eq:genTik}
	\min_\bfs\> \frac{1}{2}\|\bfL_\bfR(\bfA \bfs -\bfd)\|_2^2 +  \frac{\lambda^2}{2}\| \bfL_\bfQ(\bfs-\bfmu)\|_2^2\,,
	\end{equation}
where $\bfQ^{-1} = \bfL_\bfQ\t\bfL_\bfQ$ and $\bfR^{-1} = \bfL_\bfR\t\bfL_\bfR$.
For example, hybrid iterative methods have been investigated for the standard-form Tikhonov problem where $\bfL_\bfQ = \bfI$ in~\cite{OLeary1981, Kilmer2001,Bazan2010, gazzola2015krylov, ChNaOLe08, chung2015hybrid,hnvetynkova2009regularizing} and for the general-form Tikhonov problem in~\cite{reichel2012tikhonov,kilmer2007projection,gazzola2014generalized,hochstenbach2010iterative,hochstenbach2015golub}.  However, these previously-developed methods require $\bfL_\bfQ$ or $\bfL_\bfQ^{-1}$, which is not available in the scenarios of interest here (e.g., where $\bfQ$ represents a Mat\'{e}rn kernel or a dictionary collection).  In particular, we focus on the Mat\'{e}rn class of covariance kernels, which include among them as special cases, the exponential and the Gaussian kernel (sometimes, called squared exponential kernel). Although Mat\'{e}rn kernels represent a rich class of covariance kernels that can be adapted to the problem at hand, a main challenge is that the covariance matrix $\bfQ$ can be very large and dense, so working with its inverse or square root can be difficult. Furthermore, methods for selecting regularization parameter $\lambda$ for general-form Tikhonov are still under development, especially for large-scale problems.

\paragraph{Overview of main contributions}
In this work, we propose generalized iterative hybrid approaches for computing
approximations to~\eqref{eqn:LS_orig}.  After an appropriate change of variables, we exploit properties of the generalized Golub-Kahan bidiagonalization to develop a hybrid approach that is \emph{general} in that a rich class of covariance kernels can be incorporated, \emph{efficient} in that the main costs per iteration are matrix-vector multiplications (and not inverses and square-roots), and \emph{automatic} in that regularization parameters and stopping criteria can be determined during the iterative process. Theoretical results show that iterates of the generalized hybrid methods are equivalent to those of standard iterative methods when applied to a directly-regularized priorconditioned Tikhonov problem, after a transformation of variables.  Furthermore, we provide connections to generalized singular value decomposition (GSVD) filtered solutions for additional insight.

The paper is organized as follows. In Section~\ref{sec:Bayes}, we describe various choices for $\bfQ$,
and we describe a change of variables that {makes the computations feasible}.
To provide insight, we make connections between the transformed problem and a filtered GSVD interpretation.  However, since the GSVD is typically not available in practice, we describe in Section~\ref{sec:genlsqr} a hybrid iterative approach that is based on the generalized Golub-Kahan bidiagonalization method and that allows automatic estimation of the regularization parameter $\lambda$.  Then in Section~\ref{sec:connections} we provide some theoretical results that connect and distinguish our approach from existing methods.  Numerical results are provided in Section~\ref{sec:numerical_experiments}, and conclusions are provided in Section~\ref{sec:conclusions}.

\section{Problem setup}
\label{sec:Bayes}
We are interested in computing MAP estimates~\eqref{eqn:LS_orig} for large-scale Bayesian inverse problems.
In this section, we describe one example where this scenario may arise and then describe a change of variables that can avoid expensive matrix factorizations.  Then, we use some analytical tools such as a GSVD that is suited to our problem to provide insight and motivation for the proposed hybrid methods.

\subsection{Choice of ${\bfQ}$}
A prevalent approach in the literature models $\bfL_\bfQ$ as a sparse discretization of a differential operator (for example, the Laplacian or the biharmonic operator). This choice corresponds to a covariance matrix that represents a Gauss-Markov random field.  Since for these choices, the precision matrix (i.e., the inverse of the covariance matrix) is sparse, working with $\bfL_\bfQ$ directly has obvious computational advantages.  However, in many applications, the precision matrix is not readily available.  For example, in working with Gaussian random fields, entries of the covariance matrix are computed directly as $\bfQ_{ij} = \kappa(\bfx_i,\bfx_j)$, where $\{\bfx_i\}_{i=1}^n$ are the spatial points in the domain. There are many modeling advantages for this particular choice of $\bfQ$, but the challenge is that $\bfQ$ is often a very large and dense matrix.  However, for some covariance matrices, such as those that come from the Mat\'{e}rn family, there have been recent advancements on the efficient computation of matrix-vector products (henceforth referred to as MVPs) with $\bfQ$, and in this paper, we take advantage of these developments and develop methods that work directly with $\bfQ$. An explicit link between Gaussian fields and Gauss-Markov random fields using techniques from stochastic partial differential equations has been developed in~\cite{lindgren2011explicit}.

A popular choice for $\kappa(\cdot,\cdot)$ is from the Mat\'{e}rn family of covariance kernels~\cite{rasmussen2006gaussian}, which form an isotropic, stationary, positive-definite class of covariance kernels. We define the covariance kernel in the Mat\'{e}rn class as
\begin{equation}\label{eqn:maternfamily}
\kappa(\bfx_i,\bfx_j) = C_{\alpha,\nu}(r) = \frac{1}{2^{\nu-1}\Gamma(\nu)} \left(\sqrt{2\nu}\alpha r\right)^\nu K_\nu\left(\sqrt{2\nu}\alpha r\right)
\end{equation}
where $r=\normtwo{\bfx_i-\bfx_j}$, $\Gamma$ is the Gamma function, $K_\nu(\cdot)$ is the modified Bessel function of the second kind of order $\nu$, and $\alpha$ is a scaling factor. The choice of parameter $\nu$ in equation~\eqref{eqn:maternfamily} defines a special form for the covariance. For example, when $\nu=1/2$, $C_{\alpha,\nu}$ corresponds to the exponential covariance function, and if $\nu = 1/2+p$ where $p$ is a non-negative integer, $C_{\alpha,\nu}$ is the product of an exponential covariance and a polynomial of order $p$. Also, in the limit as $\nu\rightarrow\infty$, $C_{\alpha,\nu}$ converges to the Gaussian covariance kernel, for an appropriate scaling of $\alpha$. In Figure~\ref{fig:matern}, we provide examples of kernels from the Mat\'{e}rn covariance class, as well as realizations drawn from these covariance kernels. Another related family of covariance kernels is the $\gamma$-exponential function~\cite{rasmussen2006gaussian},
\begin{equation}
	\label{eq:gammaexp}
\kappa(r) \equiv \exp\left(-(r/\ell)^\gamma\right) \qquad 0< \gamma \leq 2.
\end{equation}

\begin{figure}[!ht]
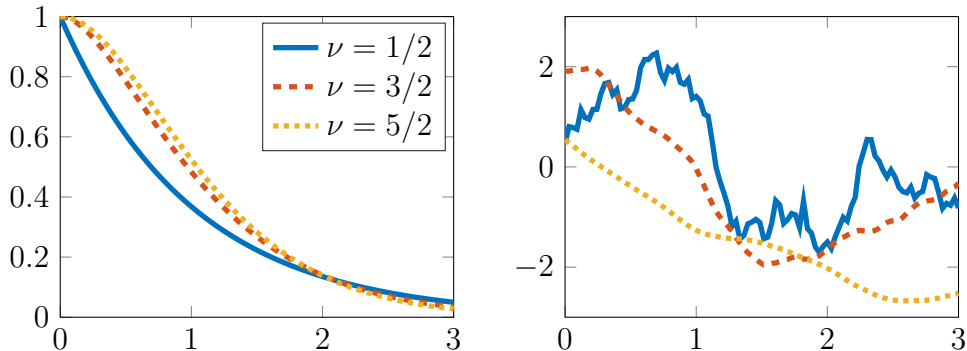

		\label{fig:matern}
	\begin{center}
		\iwidth=55mm
		\iheight=40mm
		\begin{tabular}{cc}
			\input{figs/matern1.tex} &
						\input{figs/matern2.tex}
						\end{tabular}
	\end{center}
	\caption{In the left image, we provide kernels from the Mat\'{e}rn covariance class where $\nu = 1/2, 3/2$, and $ 5/2$, and in the right image, we provide one realization drawn from each of these covariance kernels.}
	%\label{fig:dpc}
\end{figure}

In general, both the storage and the computational cost for a MVP involving a dense prior covariance matrix is $\mc{O}(n^2)$ using the naive approach.  However, significant savings are possible for stationary or translational invariant covariance kernels; note that this includes both the Mat\'{e}rn and the $\gamma$-exponential covariance families. In particular, with points located on a regular equispaced grid, one can exploit the connection between the Fast Fourier Transform (FFT) and matrices with Toeplitz structure in 1D or block-Toeplitz structure in 2D, so that the cost per MVP can be reduced to $\mc{O} (n \log n)$~\cite{nowak2003efficient}. It can also be shown that for irregular grids, the cost for approximate MVPs involving the prior covariance matrix $\bfQ$ can be reduced to $\mc{O}(n\log n)$ using Hierarchical matrices~\cite{saibaba2012efficient} or $\mc{O}(n)$ using ${\cal{H}}^2$-matrices or the Fast Multipole Method (FMM)~\cite{ambikasaran2012large}. In this work, we will only use the FFT based approach.

\subsection{Change of variables}
\label{sub:changevar}
One common approach for computing the MAP estimate~\eqref{eqn:LS_orig} is to use an iterative solver on the normal equations~\eqref{eqn:normal}.  However, each iteration would require a MVP with $\bfQ^{-1}$, or equivalently, a linear solve with $\bfQ$, which can be costly. Another standard practice is to compute a Cholesky decomposition\footnote{It is worth noting that any symmetric factorization (e.g., an eigenvalue decomposition) could be used here.  Also, our choice of notation here is consistent with the literature.} $\bfQ^{-1} = \bfL_\bfQ\t \bfL_\bfQ$ and either use an iterative solver for the general-form Tikhonov problem~\eqref{eq:genTik} or use the following change of variables,
\begin{equation*}
\bfx \>\leftarrow  \> \bfL_\bfQ (\bfs-\boldsymbol{\mu}), \qquad \bfb \> \leftarrow  \> \bfd - \bfA\bfmu
\end{equation*}
to get a standard-form Tikhonov problem,
\begin{equation}
	\label{eqn:priorcondition}
	\min_{\bfx} \> \frac{1}{2} \normtwo{\bfL_\bfR(\bfA\bfL_\bfQ^{-1}\bfx-\bfb)}^2 + \frac{\lambda^2}{2}\normtwo{\bfx}^2\,.
\end{equation}
In the literature, this particular change of variables is referred to as \emph{priorconditioning}~\cite{calvetti2005priorconditioners,calvetti2007preconditioned,calvetti2007introduction} or transformation to standard form \cite{Hansen2010}.  Although this transformation can work well for some choices of $\bfQ$, obtaining $\bfL_\bfQ$ can be prohibitively expensive for prior covariance matrices $\bfQ$ arising from discrete representations of the Mat\'{e}rn class.  Furthermore, even for cases where $\bfQ^{-1}$ is sparse, computing the Cholesky decomposition can be expensive particularly for large-scale 3D problems, as argued in~\cite{arridge2014iterated}.

In order to avoid matrix factorizations of $\bfQ$ and/or expensive linear solves with $\bfQ$, we propose a different change of variables,
\begin{equation*}
\bfx \>\leftarrow  \> \bfQ^{-1}(\bfs-\bfmu), \qquad
 \bfb \> \leftarrow  \> \bfd - \bfA\bfmu\,,
\end{equation*}
so that equation~\eqref{eqn:normal} reduces to the modified system of equations
\begin{equation}\label{eqn:normal2}
(\bfA\t \bfR^{-1} \bfA \bfQ + \lambda^2 \bfI ) \bfx = \bfA\t \bfR^{-1}\bfb\,.
\end{equation}
In summary, with this change of variables, the MAP estimate is given by $\bfs_\lambda = \bfmu + \bfQ \bfx_\lambda$, where $\bfx_\lambda$ is the solution to the following least squares (LS) problem,
\begin{equation}
	\label{eqn:LS_tik_x}
	\min_\bfx \> \frac{1}{2}\|\bfA \bfQ \bfx -\bfb\|^2_{\bfR^{-1}} +\frac{ \lambda^2}{2}\|\bfx\|^2_\bfQ\,.
	\end{equation}
It is important to observe that the above change of variables removes both inverses and square roots with respect to $\bfQ$. Furthermore, in the Bayesian interpretation, we now have
\[ \bfb|\bfx \sim \mc{N}(\bfA\bfQ\bfx,\bfR^{-1}) \qquad \bfx \sim \mc{N}(\bfzero,\lambda^{-2}\bfQ^{-1}).\]

In Section~\ref{sec:genlsqr}, we describe an efficient hybrid iterative method for computing a solution to the transformed problem~\eqref{eqn:LS_tik_x}, where the regularization parameter $\lambda$ can be selected adaptively.  However, to provide some insight and motivation for the hybrid method, in the next subsection, we first use some tools from numerical linear algebra to analyze and interpret the MAP estimate as a filtered GSVD solution.

\subsection{Interpretation and analysis using the GSVD}\label{s_gsvd}
We begin by reviewing the GSVD and showing that the MAP estimate, $\bfs_\lambda = \bfmu+ \bfQ \bfx_\lambda$ where $\bfx_\lambda$ is the solution to~\eqref{eqn:LS_tik_x}, can be expressed in terms of a filtered GSVD solution. We use an alternate definition of the GSVD that was proposed in~\cite{van1976generalizing}, which is less commonly used in the inverse problems literature but relevant to our discussion.

Given two positive definite matrices $\bfQ$ and $\bfR$, Theorem 3 in~\cite{van1976generalizing} guarantees that there exist $\bfU_\bfR\in \mathbb{R}^{m\times m}$ satisfying $\bfU_\bfR\t\bfR^{-1}\bfU_\bfR = \bfI_m$ and $\bfV_\bfQ \in \mathbb{R}^{n\times n}$ satisfying $\bfV_\bfQ\t \bfQ^{-1} \bfV_\bfQ = \bfI_n$, such that
\begin{equation}
	\label{gsv}
	\bfU_\bfR^{-1}\bfA\bfV_\bfQ = \boldsymbol\Sigma_{\bfR\bfQ}\qquad \boldsymbol\Sigma_{\bfR\bfQ} = \diag{\hat{\sigma}_1,\dots,\hat{\sigma}_n}\,,
\end{equation} where the generalized singular values $\hat{\sigma}_j$ satisfy the following variational form,
\begin{equation}\label{e_gsv}
\hat{\sigma}_{j}(\bfA)  \equiv \min_{\{\mc{S}: \> \text{dim}\> \mc{S} = n - j + 1\}} \max_{\{\bfx:\> \bfx \neq \boldsymbol{0}, \bfx \in \mc{S} \}} \frac{\norm{\bfA\bfx}{\bfR^{-1}}}{\norm{\bfx}{\bfQ^{-1}}},
\end{equation}
where $\mc{S}$ is a subspace of $\mb{R}^n$.
Our analysis exploits the relationship between the GSVD of $\bfA$ and the SVD of priorconditioned matrix
$\widehat{\bfA} \equiv \bfL_\bfR \bfA\bfL_\bfQ^{-1}$, as summarized in the following result.
\begin{prop}
The generalized singular values of $\bfA$ defined in~\eqref{e_gsv} are the singular values of $\widehat{\bfA} \equiv \bfL_\bfR \bfA\bfL_\bfQ^{-1}$, and the left and right singular vectors of $\widehat{\bfA}$ are given by $\widehat{\bfU} = \bfL_\bfR \bfU_\bfR$ and $\widehat{\bfV} = \bfL_\bfQ^{-\top} \bfV_\bfQ^{-\top}$ respectively.
\end{prop}
\begin{proof}
Straightforward multiplication reveals that $\widehat\bfA = \widehat \bfU \bfSigma_{\bfR \bfQ}\widehat \bfV\t$, where orthogonal matrices $\widehat \bfU$ and $\widehat \bfV$ are defined as above.  Another proof of equivalence for the singular values is to use the change of variables $\boldsymbol\zeta \leftarrow \bfL_\bfQ\bfx$ and the fact that $\norm{\bfA\bfx}{\bfR^{-1}} = \normtwo{\bfL_\bfR\bfA\bfx}$ in~\eqref{e_gsv} to get
\begin{equation}
\hat{\sigma}_{j}(\bfA)  = \min_{\{\mc{S}: \> \text{dim}\> \mc{S} = n - j + 1\}} \max_{\{\boldsymbol\zeta :\> \boldsymbol\zeta \neq 0, \boldsymbol\zeta \in \mc{S} \}} \frac{\normtwo{\bfL_\bfR \bfA\bfL_\bfQ^{-1}\boldsymbol\zeta}}{\normtwo{\boldsymbol\zeta}},
\end{equation}
where the conclusion follows from the variational characterization of singular values; see~\cite[Theorem 7.3.8]{horn2012matrix}.
\end{proof}

If $\bar\bfu_j$ and $\bar\bfv_j$ denote the columns of $\bfU_\bfR$ and $\bfV_\bfQ$ respectively, then it can be shown that the MAP solution can be written as a filtered GSVD solution,
\begin{equation}
	\label{eq:filtGSVD}
\bfs_\lambda \> = \> \bfmu + \sum_{j=1}^n \hat\phi_j \frac{\bar\bfu_j\t\bfR^{-1}\bfb}{\hat{\sigma}_j} \bar\bfv_j \,,
\end{equation}
where the generalized Tikhonov filter factors are given by
\[ \hat\phi_j \> \equiv \>\frac{\hat{\sigma}_j^2}{ \hat{\sigma}_j^2 + \lambda^2} \qquad j=1,\dots,n. \]
Other spectral filtering techniques can be considered as well, such as the
truncated GSVD, where for some $1 \leq k \leq n$, the generalized filter factors are given by,
\[ \hat\phi_j \> \equiv \> \left\{ \begin{array}{cc} 1 & 1 \leq j \leq k \\ 0 & j >k \end{array}\right.  \qquad j=1,\dots,n \,.\]
Note that the MAP solution can be alternatively written as
\begin{equation}\label{e_map_gsvd}
\bfs_\lambda \> = \> \bfmu + {\bfV}_\bfQ \B{\Phi\Sigma}^{\dagger}_{\bfR \bfQ}\bfg\,,
\end{equation}
where $\bfg \equiv {\bfU}^T_\bfR\bfR^{-1}\bfb$, and $\B{\Phi} = \diag{\hat\phi_1, \dots, \hat\phi_n}$.

A new equivalence relationship between the solution from our proposed generalized hybrid approach and a filtered GSVD solution will be proved in Section~\ref{sec:connections}, so, for completeness, we provide a brief discussion regarding the discrete Picard plot for GSVD and point the interested reader to \cite{Hansen2010} for more details.
We use a 1D inverse heat example from \cite{hansen1994regularization} and define $\bfmu=\bfzero$, $\bfR = \bfI$, and prior covariance matrix $\bfQ$ to represent a kernel from the Mat\`{e}rn covariance family~\eqref{eq:gammaexp} where $\nu = 1/2$ and $\alpha=2$.
Discrete Picard plots corresponding to the SVD and GSVD of $\bfA$ are provided in Figure~\ref{fig:dpc}, where the generalized singular values, $\hat\sigma_j$, decay faster than the singular values of $\bfA$.  Although such analytical tools are not feasible for realistic problems of interest since computing the GSVD is prohibitive, the significance of this observation is that we can expect iterative methods applied to the prior-conditioned problem to converge faster, if the singular values of $\widehat\bfA$ decay faster than those of $\bfA$.
In the next section, we derive an efficient Krylov subspace solver for estimating solutions to~\eqref{eqn:LS_tik_x} that only requires MVPs and allows for automatic regularization parameter selection.

\begin{figure}[tb!]
	\begin{center}
		\iwidth=120mm
		\iheight=50mm
			\input{figs/dpc.tex}
	\end{center}
	\caption{Discrete Picard plots.	Notice that the generalized singular values decay faster than the singular values of $\bfA$.  Absolute values of the SVD coefficients, $|\bfu_j\t \bfb|$, and GSVD coefficients, $|\bar\bfu_j\t \bfb|,$ decay initially, but stabilize at the noise level (here, $10^{-6}$).  Absolute values of the solution coefficients decrease initially but eventually increase due to noise contamination.
	}
	\label{fig:dpc}
\end{figure}

\section{Iterative methods based on generalized Golub-Kahan bidiagonalization}\label{sec:genlsqr}
In this section, we describe an iterative solver for LS problem~\eqref{eqn:LS_tik_x} that is based on the generalized Golub-Kahan (gen-GK) bidiagonalization \cite{arioli2013generalized}. Such iterative methods are desirable for problems where matrices $\bfA$ and $\bfQ$ may be so large that they cannot be explicitly stored, but they can be accessed via function calls.  We first describe the gen-GK process and highlight relevant properties and connections.  Then, we describe how to solve~\eqref{eqn:LS_tik_x} efficiently by exploiting the gen-GK relationships for the case where $\lambda$ is provided a priori.  However, for problems where a good choice is not available, we describe a hybrid approach where sophisticated regularization parameter selection methods can be used on the projected problem to estimate $\lambda$ at each iteration.  Similar to previously-studied hybrid methods, our approach assumes that solutions can be captured in relatively few iterations or that preconditioning can be used so that the required number of iterations remains small.  This is typically the case for ill-posed problems, but even more relevant here due to the GSVD analysis in Figure~\ref{fig:dpc}, where the generalized singular values tend to decay faster than the singular values of $\bfA$.

\subsection{Generalized Golub-Kahan bidiagonalization}

The discussion in this section is closely related to the approach in~\cite{arioli2013generalized}. The major difference is that the generated Krylov subspace is different due to the change of variables described in Section~\ref{sub:changevar}.

Given matrices $\bfA$, $\bfR$, $\bfQ$, and vector $\bfb,$ with initializations $\beta_1 = \norm{\bfb}{\bfR^{-1}}, \bfu_1 = \bfb/\beta_1$ and $\alpha_1 \bfv_1 = \bfA\t \bfR^{-1} \bfu_1$, the $k$th iteration of the gen-GK bidiagonalization procedure generates vectors $\bfu_{k+1}$ and $\bfv_{k+1}$ such that
\begin{align*}
	\beta_{k+1} \bfu_{k+1} & = \bfA \bfQ \bfv_k -\alpha_k \bfu_k\\
		\alpha_{k+1} \bfv_{k+1} & = \bfA\t \bfR^{-1} \bfu_{k+1} -\beta_{k+1} \bfv_k,
	\end{align*}
where scalars $\alpha_i, \beta_i \geq 0$ are chosen such that $\norm{\bfu_i}{\bfR^{-1}} = \norm{\bfv_i}{\bfQ} = 1$. At the end of $k$ steps, we have
\[ \bfB_k \equiv \> \begin{bmatrix}
\alpha_1 \\ \beta_2 & \alpha_2 \\ & \beta_3 & \ddots \\ & & \ddots & \alpha_k \\ & & & \beta_{k+1}
\end{bmatrix}\,,  \qquad \bfU_{k+1} \equiv [\bfu_1,\dots,\bfu_{k+1}],\quad \mbox{and} \quad \bfV_k \equiv [\bfv_1,\dots,\bfv_k],\]
where the following relations hold up to machine precision,
\begin{align}\label{e_bk}
\bfU_{k+1}\beta_1 \bfe_1  =  &\> \bfb \\ \label{e_vk}
\bfA \bfQ \bfV_k = & \>\bfU_{k+1} \bfB_k \\ \label{e_uk}
\bfA\t \bfR^{-1} \bfU_{k+1} = & \> \bfV_k \bfB_k\t + \alpha_{k+1}\bfv_{k+1}\bfe_{k+1}\t\,.
\end{align}
Furthermore, in exact arithmetic, matrices $\bfU_{k+1}$ and $\bfV_k$ satisfy the following orthogonality conditions
\begin{equation}
	\label{eq:orthog}
	\bfU_{k+1}\t \bfR^{-1} \bfU_{k+1} = \bfI_{k+1} \qquad \mbox{and} \qquad \bfV_k\t \bfQ \bfV_k = \bfI_k.
\end{equation}

An algorithm for the gen-GK bidiagonalization process is provided in Algorithm~\ref{alg:wlsqr}. In addition to MVPs with $\bfA$ and $\bfA\t$ that are required for the standard GK bidiagonalization \cite{GoKa65}, each iteration of gen-GK bidiagonalization requires two MVPs with $\bfQ$ and two solves with $\bfR$ (which are simple); in particular, we emphasize that Algorithm~\ref{alg:wlsqr} avoids $\bfQ^{-1}$ and $\bfL_\bfQ$, due to the change of variables described in Section~\ref{sub:changevar}.

\begin{algorithm}[!h]
\begin{algorithmic}[1]
\REQUIRE Matrices $\bfA$, $\bfR$ and $\bfQ$, and vector $\bfb$.
\STATE $\beta_1 \bfu_1 = \bfb,$ where $\beta_1 = \norm{\bfb}{\bfR^{-1}}$
\STATE $\alpha_1 \bfv_1 = \bfA\t \bfR^{-1}\bfu_1$
\FOR {i=1, \dots, k}
\STATE $\beta_{i+1}\bfu_{i+1} = \bfA\bfQ\bfv_i - \alpha_i \bfu_i$, where $\beta_{i+1} = \norm{\bfA\bfQ\bfv_i - \alpha_i \bfu_i}{\bfR^{-1}}$
\STATE $\alpha_{i+1}\bfv_{i+1} = \bfA\t \bfR^{-1} \bfu_{i+1} - \beta_{i+1} \bfv_i$, where $\alpha_{i+1} = \norm{\bfA\t \bfR^{-1} \bfu_{i+1} - \beta_{i+1} \bfv_i}{\bfQ}$
\ENDFOR
\end{algorithmic}
\caption{generalized Golub-Kahan (gen-GK) bidiagonalization}
\label{alg:wlsqr}
\end{algorithm}

\paragraph{Remarks on gen-GK bidiagonalization}
Recall that the $k$-dimensional Krylov subspace associated with matrix $\bfC$ and vector $\bfg$ is defined as
\[\krylov{k}(\bfC,\bfg) \equiv \text{Span}\{\bfg,\bfC\bfg,\dots,\bfC^{(k-1)} \bfg \}.\]
It is well known that at the $k$--th iteration of the standard GK bidiagonalization process, orthonormal columns $\bfv_i, i = 1, ..., k$ span the $k$-dimensional Krylov subspace $\krylov{k}(\bfA\t\bfA,\bfA\t\bfb)$ \cite{GoKa65}.  For gen-GK, it can be shown using the relations in~\eqref{e_bk}-\eqref{e_uk} that the columns of $\bfU_k$ and $\bfV_k$  respectively form  $\bfR^{-1}$-orthogonal and $\bfQ$-orthogonal (c.f. Equation~\eqref{eq:orthog}) bases for the following Krylov subspaces:
\begin{align}
\Span{\bfU_k} \> = & \> \krylov{k}(\bfA\bfQ\bfA\t\bfR^{-1}, \bfb)\\
\Span{\bfV_k} \> = & \> \krylov{k}(\bfA\t \bfR^{-1}\bfA\bfQ, \bfA\t\bfR^{-1}\bfb)
\end{align}
respectively.
A useful property of Krylov subspaces is that they are shift-invariant, i.e., $\krylov{k}(\bfC,\bfg) = \krylov{k}(\bfC+\lambda^2 \bfI,\bfg)$ for any $\lambda$. In particular, this implies that
\[ \krylov{k}(\bfA\t\bfR^{-1}\bfA\bfQ, \bfv_1) = \krylov{k}(\bfA\t\bfR^{-1}\bfA\bfQ + \lambda^2\bfI, \bfv_1),\]
and consequently by combining~\eqref{e_vk} and~\eqref{e_uk}, we get
\begin{align}\nonumber
 \left(\bfA\t\bfR^{-1}\bfA\bfQ + \lambda^2 \bfI_n \right)\bfV_k = & \quad \bfV_k(\bfB_k\t\bfB_k + \lambda^2 \bfI_k) + \alpha_{k+1}\bfv_{k+1}\bfe_k\t \bfB_k \\ \label{eqn:shift}
 = &  \quad \bfV_{k+1}\left(\bar{\bfB}_k + \lambda^2\begin{bmatrix}\bfI_k \\ \bfzero_{1\times k}\end{bmatrix}\right),
\end{align}
where
\begin{equation}
	\label{eqn:lsmr}
	\bar{\beta}_k  \equiv \beta_k\alpha_k \qquad \mbox{and} \qquad \bar{\bfB}_k\equiv \>  \begin{bmatrix} \bfB_k\t \bfB_k\\ \bar{\beta}_{k+1}\bfe_k\t\end{bmatrix}.
	\end{equation}
We will see in Section~\ref{s_param} that the shift-invariant property of Krylov subspaces is essential for choosing regularization parameters reliably and efficiently during the iterative process.  In fact, our approach creates a basis for the Krylov subspace $\krylov{k}(\bfA\t\bfR^{-1}\bfA\bfQ, \bfv_1)$ that is independent of $\lambda$, so that the choice of regularization parameter can be done adaptively.  This distinguishes us from approaches that use iterative Krylov methods to solve the equivalent general-form Tikhonov problem,
\begin{equation*}
	\min_\bfs\> \frac{1}{2}\left\|\begin{bmatrix}
		\bfL_\bfR\bfA \\ \lambda \bfL_\bfQ
	\end{bmatrix} \bfs - \begin{bmatrix}
		\bfL_\bfR \bfd \\ \bfL_\bfQ \bfmu
	\end{bmatrix}\right\|_2^2\,,
\end{equation*}
in which case, the Krylov subspace depends on the choice of $\lambda$ and the regularizing iteration property is not straightforward \cite{Hansen2010,HaHa93}.  Hybrid iterative methods could be used to solve the priorconditioned problem~\eqref{eqn:priorcondition}, but that approach would require $\bfL_\bfQ.$

\subsection{Solving the LS problem}
In this section, we seek approximate solutions to~\eqref{eqn:LS_tik_x} by using the gen-GK relations above to obtain a sequence of projected LS problems.  For clarity of presentation in this subsection, we assume $\lambda$ is fixed and drop the subscript.  In particular, we seek solutions of the form $\bfx_{k}  = \bfV_k \bfz_{k} $, so that
\[ \bfx_k \in \text{Span}\{\bfV_k\} = \krylov{k}(\bfA\t\bfR^{-1}\bfA\bfQ,\bfA\t\bfR^{-1} \bfb) \equiv \mc{S}_k. \]
Define the residual at step $k$ as $\bfr_k \equiv \>  \bfA\bfQ\bfx_k-\bfb$. It follows from Equations~\eqref{e_bk}-\eqref{e_uk} that
{\[ \bfr_k \equiv \>  \bfA\bfQ\bfx_k-\bfb = \bfU_{k+1} \left( \bfB_k \bfz_k - \beta_1 \bfe_1\right), \]}
and furthermore,
\[ \bfA\t\bfR^{-1}\bfr_k  = \bfV_{k+1} (\bar{\bfB}_k\bfz_k - \bar{\beta}_1\bfe_1).\]
We consider two approaches for obtaining coefficients $\bfz_k$.  We propose to either take $\bfz_k$ that minimizes the gen-LSQR problem,
\begin{equation}\label{e_wlsqr}\min_{\bfx_k \in \mc{S}_k } \>\frac{1}{2}\norm{\bfr_{k}}{\bfR^{-1}}^2+\frac{\lambda^2}{2}\norm{\bfx_{k}}{\bfQ}^2\quad \Leftrightarrow \quad \min_{\bfz_k \in \mb{\bfR}^k} \> \frac{1}{2}\normtwo{ \bfB_k \bfz_k-\beta_1 \bfe_1 }^2 + \frac{\lambda^2}{2} \normtwo{\bfz_k}^2,
\end{equation}
or take $\bfz_k$ to minimize the gen-LSMR problem,
\begin{equation}\label{e_wlsmr} \min_{\bfx_k \in \mc{S}_k} \>\frac{1}{2}\norm{\bfA\t\bfR^{-1}\bfr_{k}}{\bfQ}^2+\frac{\lambda^2}{2}\norm{\bfx_{k}}{\bfQ}^2\quad \Leftrightarrow \quad \min_{\bfz_k\in \mb{R}^k } \> \frac{1}{2}\normtwo{  \bar{\bfB}_k\bfz_k -\bar{\beta}_1 \bfe_1 }^2 + \frac{\lambda^2}{2} \normtwo{\bfz_k}^2\,,  \end{equation}
where the gen-GK relations were used to obtain the equivalences.
These approaches are motivated by standard LSQR \cite{paige1982lsqr,PaSa82b} and LSMR \cite{fong2011lsmr}.

After computing a solution to the projected problem, an approximate solution to the original problem~\eqref{eqn:LS_orig} can be recovered by undoing the change of variables,
\begin{equation}\label{eqn:undo_change}
\bfs_{k} =  \bfmu + \bfQ \bfx_{k}  = \bfmu + \bfQ\bfV_k\bfz_{k}\,,
\end{equation}
where, now, $\bfs_{k} \in \bfmu + \bfQ\mc{S}_k$.

\subsection{A generalized hybrid approach}\label{s_param}
Thus far, we have assumed that regularization parameter $\lambda$ is provided a priori.  However, obtaining a good regularization parameter can be difficult, especially for extremely large-scale problems, where it may be necessary to solve many systems for different $\lambda$ values.  In this section, we take advantage of the shift-invariance property of Krylov subspaces and propose a generalized hybrid approach, where regularization parameters can be estimated during the iterative process.

The basic idea is to use well-studied, sophisticated regularization parameter selection schemes \cite{Hansen2010,HaHa93,Vogel2002} on projected problems~\eqref{e_wlsqr} and~\eqref{e_wlsmr}, where for small values of $k,$ matrices $\bfB_k$ and $\bar{\bfB}_k$ are only of size $(k+1) \times k.$  This approach is not novel, with previous work on parameter selection within hybrid LSQR and LSMR methods including \cite{ChNaOLe08,chung2015hybrid,Kilmer2001,gazzola2015krylov,renaut2010regularization}. However, to the best of our knowledge, no one has looked at hybrid methods based on the generalized Golub-Kahan bidiagonalization, where the novelty is that we propose a hybrid variant that is ideal for problems where $\bfQ$ is modeled and applied directly, without the need for its inverse or square root.

Although a wide range of regularization parameter methods can be used in our framework, in this paper we consider the generalized cross validation (GCV) approach, the discrepancy principle (DP), and the unbiased predictive risk estimator (UPRE) approach. As a benchmark, all of the computed parameters will be compared to the optimal regularization parameter $\lambda_{\rm opt}$, which minimizes the $2$-norm of the error between the reconstruction and the truth. In the following, we describe GCV as a means to select regularization parameters and refer the reader to the Appendix for details regarding the use of DP and UPRE in a hybrid framework.

GCV \cite{GoHeWa79} is a statistical technique that is based on a leave-one-out approach.  The GCV parameter, $\lambda_{\rm gcv}$, is selected to minimize the GCV function corresponding to the general-form Tikhonov problem~\eqref{eq:genTik},
\begin{equation}\label{e_gcv}
G(\lambda) = \> \frac{n \norm{ \bfA \bfs_\lambda - \bfd}{\bfR^{-1}}^2}{\left[\trace(\bfI_m-\bfL_\bfR \bfA \bfA_\lambda^{\dagger})\right]^2}\,,
\end{equation}
where $\bfA_\lambda^{\dagger} = (\bfA\t\bfR^{-1}\bfA + \lambda^2 \bfQ^{-1})^{-1}\bfA\t \bfL_\bfR\t$. We have assumed $\bfmu=\bfzero$ for simplicity.  The GCV function can be simplified by using the GSVD defined in Section~\ref{s_gsvd} as
\begin{equation*}
G(\lambda) = \> \frac{n \left( \displaystyle \sum_{i=1}^n \left(\frac{\lambda^2 g_i}{\hat\sigma_i^2 + \lambda^2}\right)^2 + \sum_{i=n+1}^m g_i^2 \right)}
{\displaystyle \left(m-n+ \sum_{i=1}^n \frac{\lambda^2}{\hat{\sigma}_i^2 + \lambda^2} \right)^2}\,,
\end{equation*}
where $\hat{\sigma}_i$ are the generalized singular values and $g_i$ is the $i$th element of $\bfg = \bfU_\bfR\t \bfR^{-1}\bfb$. Since the GSVD is often not available in practice, a general approach is to use the GCV function corresponding to the projected problem~\eqref{e_wlsqr}, where the GCV parameter at the $k$th iteration minimizes,
\begin{equation}\label{e_gcv_proj}
 G_\text{proj}(\lambda) \equiv \> \frac{k \normtwo{(\bfI - \bfB_k \bfB_{k,\lambda}^\dagger)\beta_1\bfe_1}^2}{\left[\trace(\bfI_{k+1} - \bfB_k\bfB_{k,\lambda}^\dagger)\right]^2 },
\end{equation}
where $\bfB_{k,\lambda}^\dagger = (\bfB_k\t\bfB_k+\lambda^2 \bfI)^{-1} \bfB_k\t$.
A weighted-GCV (WGCV) approach \cite{ChNaOLe08} has been suggested for use within hybrid methods, where a weighting parameter is introduced in the denominator of~\eqref{e_gcv_proj}.  We denote $\lambda_{\rm wgcv}$ to be the regularization parameter computed using WGCV. Recent insights \cite{renaut2015hybrid} into the choice of weighting parameter as well as connections between the GCV function for the original versus projected problem extend naturally to this framework.

We also remark that analogous methods can be derived for gen-LSMR, where the projected problem is given in~\eqref{e_wlsmr}. In this case, we replace $\bfB_k$ with $\widehat{\bfB}_k$, and $\beta_1$ with $\bar{\beta}_1$, which yield equivalent expressions for the projected problem. This will not be discussed in great detail and the interested reader is referred to~\cite{chung2015hybrid}.

Finally, we briefly comment on stopping criteria for the gen-GK process in a generalized hybrid approach.  In particular, we follow the approaches described in \cite{ChNaOLe08,chung2015hybrid}, where the iterative process is terminated if a maximum number of iterations is attained, a GCV function defined in terms of the iteration, $G(k)$, attains a minimum or flattens out, or tolerances on the residual are attained. Other parameter estimation techniques yield similar stopping criteria.

\section{Theory and connections} % (fold
\label{sec:connections}

\subsection{Interpreting gen-LSQR iterates as filtered GSVD solutions}

In Section~\ref{s_gsvd}, we used the GSVD to express the MAP estimate, $\bfs_\lambda$, as a filtered GSVD solution. In the following, we exploit the connection between gen-LSQR iterates and filtered GSVD solutions, following the approach presented in~\cite{jensen2007iterative}, to provide insight into the suitability of gen-LSQR for computing regularized solutions.

Let $\lambda$ be fixed. Then using the GSVD~\eqref{gsv}, we can express the Krylov subspace for the solution at the $k$--th iteration as
\[ \krylov{k}(\bfA^\top\bfR^{-1}\bfA\bfQ + \lambda^2 \bfI_n,\bfA^\top\bfR^{-1}{\bfb}) = \Span{{\bfV}_\bfQ^{-\top}({\bfSigma}_{\bfR\bfQ}\t {\bfSigma}_{\bfR\bfQ} + \lambda^2 \bfI_n)^{q}\B{\Sigma}_{\bfR\bfQ}\t\bfg}_{q=0}^{k-1}\,. \]
Thus, the $k$th iterate of gen-LSQR can be expressed as
\[\bfx_{k} =  \bfQ^{-1}{\bfV}_\bfQ{\Phi}_k\boldsymbol\Sigma^{\dagger}_{\bfR\bfQ}\bfg, \qquad \mbox{where} \qquad \boldsymbol{\Phi}_k \equiv {\mc{P}}_k(\bfSigma_{\bfR\bfQ}\t \bfSigma_{\bfR\bfQ}+\lambda^2\bfI_n)\bfSigma_{\bfR\bfQ}\t \bfSigma_{\bfR\bfQ}, \]
$ {\mc{P}}_k(\cdot)$ is a polynomial of degree $\leq k-1$ and $\bfV_\bfQ^{-\top} = \bfQ^{-1}\bfV_\bfQ$.  Thus, the approximation to the MAP estimate at the $k$--th iteration is given as
\begin{equation}\label{eqn:krylov_trans} \bfs_{k} = \B{\mu} +  {\bfV}_\bfQ\B{\Phi}_k\boldsymbol\Sigma^{\dagger}_{\bfR\bfQ}\bfg\,.\end{equation}
In summary, we have shown that the MAP approximations generated by gen-LSQR can be expressed as filtered GSVD solutions, where the filter factors are defined by the polynomial $\mc{P}_k$. Compare this with the filtered GSVD solution in~\eqref{e_map_gsvd}. From the analysis in Section~\ref{s_gsvd} (in particular, see Figure~\ref{fig:dpc}), we expect the problem to satisfy the discrete Picard condition. Therefore, in general, we expect the early gen-LSQR iterates to have a better representation of the generalized singular vectors corresponding to the dominant generalized singular values.  Although the gen-LSMR solution is obtained using a different subproblem, gen-LSMR solutions lie within the same subspace as gen-LSQR solutions, and thus, can also be interpreted as filtered GSVD solutions but with different filter factors.

\subsection{Connection to iterative methods with $\bfQ$-orthogonality}
Next, for fixed $\lambda$, we establish equivalences
between our proposed solvers and pre-existing solvers, namely Conjugate Gradients (CG) and MINRES with $\bfQ$-orthogonality for solving~\eqref{eqn:normal2}\cite{saad2003iterative}. Before proceeding, notice that $ \bfM \equiv \bfA\t \bfR^{-1} \bfA \bfQ$ is symmetric with respect to the $\bfQ$-inner product, $\langle \bfx,\bfy\rangle_\bfQ = \bfy\t \bfQ \bfx$. This follows since
	\[ \langle \bfM \bfx,\bfy\rangle_\bfQ = \bfy\t \bfQ \bfA\t \bfR^{-1} \bfA \bfQ \bfx = \langle \bfx,\bfM \bfy\rangle_\bfQ. \]
This observation is crucial in the following theorem.

\begin{theorem}\label{t_lsqr_cg} The gen-LSQR solution after $k$ iterations (i.e., minimizer of~\eqref{e_wlsqr}) is mathematically equivalent to the solution obtained by performing $k$ steps of CG with $\bfQ$-orthogonality on~\eqref{eqn:normal2};
similarly, the $k$th gen-LSMR solution (i.e., minimizer of~\eqref{e_wlsmr}) is mathematically equivalent to the solution obtained by performing $k$ steps of MINRES with $\bfQ$-orthogonality on~\eqref{eqn:normal2}. \end{theorem}
\begin{proof}
First, notice that at the $k$--th step of CG with $\bfQ$-inner products applied to~\eqref{eqn:normal2}, we have matrices $\bfV_k^\text{L}$ and $\bfT_k$ that satisfy the following Lanczos relationship
\begin{equation}
	\label{eq:Lan}
	\bfM\bfV_k^\text{L} = \bfV_{k+1}^\text{L}\begin{bmatrix} \bfT_k \\ \beta_{k+1}^\text{L}\bfe_k^\top \end{bmatrix} \equiv \bfV_{k+1}^\text{L} \bar{\bfT}_k,
\end{equation}
where $(\bfV_k^\text{L})^\top\bfQ\bfV_k^\text{L} = \bfI_k$ and $\bfT_k$ is a symmetric tridiagonal matrix. On the other hand, after $k$ steps of gen-GK bidiagonalization
 (by combining~\eqref{e_vk} and~\eqref{e_uk}), we have
\begin{equation}
	\label{eq:GK}
	\bfM\bfV_k = \bfV_{k+1} \begin{bmatrix}\bfB_k^\top\bfB_k\\ \bar{\beta}_{k+1}\bfe_k^\top\end{bmatrix} = \bfV_{k+1} \bar\bfB_k.
\end{equation}
Next observe that both CG with $\bfQ$-orthogonality and gen-LSQR generate a basis for the same subspace, i.e., \[\Span{\bfV_k^\text{L}} = \Span{\bfV_k} = \krylov{k}(\bfM, \bfA^\top\bfR^{-1}\bfb).\]
Therefore, there exists an orthogonal matrix $\bfZ$~\cite[Equation (4.2.7)]{meyer2000matrix} such that $\bfV_k^\text{L} = \bfV_k\bfZ.$
Since
\[ (\bfV_k^\text{L})^\top\bfQ\bfM\bfV_k^\text{L} = \bfT_k \qquad \mbox{and} \qquad \bfV_k^\top\bfQ\bfM\bfV_k = \bfB_k^\top\bfB_k,\]
we can readily see that $\bfT_k = \bfZ^\top\bfB_k^\top\bfB_k\bfZ$; in other words, $\bfT_k$ is similar to $\bfB_k^\top\bfB_k$.

The approximate solution after $k$ steps of CG with $\bfQ$-inner products can be expressed as
\[ \bfx_k^\text{CG} = \bfV_k^\text{L}(\bfT_k + \lambda^2 \bfI_k)^{-1} \beta_1 \bfe_1 =\bfV_k^\text{L}(\bfT_k + \lambda^2 \bfI_k)^{-1}(\bfV_k^\text{L})^\top\bfQ\bfA^\top\bfR^{-1}\bfb. \]
Similarly, the approximate solution after $k$ steps of gen-LSQR can be expressed as
\[ \bfx_k^\text{gen-LSQR} = \bfV_k(\bfB_k^\top\bfB_k + \lambda^2 \bfI_k)^{-1}\bfB_k^\top\beta_1 \bfe_1 = \bfV_k(\bfB_k^\top\bfB_k + \lambda^2 \bfI_k)^{-1}\bfV_k^\top\bfQ\bfA^\top\bfR^{-1}\bfb. \]
Since
\begin{align*}
\bfV_k^\text{L}(\bfT_k + \lambda^2 \bfI_k)^{-1}(\bfV_k^\text{L})^\top = & \> \bfV_k\bfZ(\bfZ^\top\bfB_k^\top\bfB_k\bfZ + \lambda^2 \bfZ^\top\bfZ)^{-1} \bfZ^\top\bfV_k^\top \\
	=& \>   \bfV_k(\bfB_k^\top\bfB_k + \lambda^2 \bfI_k)^{-1} \bfV_k^\top\,,
\end{align*}
we conclude that $ \bfx_k^\text{CG} = \bfx_k^\text{gen-LSQR}$; i.e., both algorithms generate the same sequence of iterates in exact arithmetic.

Our next goal is to prove the equivalence of iterates between gen-LSMR and MINRES with $\bfQ$-inner products.
After $k$ steps of MINRES with $\bfQ$-inner products, the approximate solution can be expressed as
\[ \bfx_k^\text{MINRES} = \bfV_k^\text{L}(\bar{\bfT}_k^\top\bar{\bfT}_k + \lambda^2 \bfI_k)^{-1}\bar{\bfT}_k^\top\bar{\beta}_1\bfe_1 = \bfV_k^\text{L}(\bar{\bfT}_k^T\bar{\bfT}_k + \lambda^2 \bfI_k)^{-1}(\bfV_k^\text{L})^\top\bfQ\bfM\bfA^\top\bfR^{-1}\bfb. \]
Similarly, after $k$ steps of gen-LSMR, the approximate solution can be expressed as
\[ \bfx_k^\text{gen-LSMR} = \bfV_k(\bar{\bfB}_k^\top\bar{\bfB}_k + \lambda^2 \bfI_k)^{-1}\bar{\bfB}_k^\top\bar{\beta}_1 \bfe_1 = \bfV_k(\bar{\bfB}_k^T\bar{\bfB}_k + \lambda^2 \bfI_k)^{-1}\bfV_k^\top\bfQ\bfM\bfA^\top\bfR^{-1}\bfb. \]
Next, using the relations in~\eqref{eq:Lan} and~\eqref{eq:GK}, it is easy to see that
\[ (\bfV_k^\text{L})^\top\bfM^\top\bfQ\bfM\bfV_k^\text{L} =  \bar{\bfT}_k^\top\bar{\bfT}_k \qquad \mbox{and} \qquad \bfV_k^T\bfM^\top\bfQ\bfM\bfV_k = \bar{\bfB}_k^\top\bar{\bfB}_k\,, \]
so $\bar{\bfT}_k^\top\bar{\bfT}_k = \bfZ^T \bar{\bfB}_k^\top\bar{\bfB}_k \bfZ$.  Thus
\[ \bfV_k^\text{L}(\bar{\bfT}_k^\top\bar{\bfT}_k + \lambda^2 \bfI_k)^{-1}(\bfV_k^\text{L})^\top = \bfV_k(\bar{\bfB}_k^\top\bar{\bfB}_k + \lambda^2 \bfI_k)^{-1}\bfV_k^\top.\]
Therefore,  $ \bfx_k^\text{MINRES} = \bfx_k^\text{gen-LSMR}$ and this concludes the proof.
\end{proof}

Although we have established equivalence among the methods in exact arithmetic, the significance of the above result is that, similar to works comparing CG and LSQR \cite{paige1982lsqr,PaSa82b}, we suggest that in a computational context, one should use gen-LSQR and gen-LSMR and hybrid variants, as they are numerically stable and avoid squaring the condition number (e.g., see discussion in~\cite[Section 8.1]{saad2003iterative}).

\subsection{Equivalence result for generalized hybrid iterates}
Lastly, we justify the generalized hybrid approach, which can be interpreted as a ``project-then-regularize'' approach since regularization parameters are selected on the projected problem. In particular, for fixed $\lambda$ and in exact arithmetic, we show an equivalence between gen-LSQR iterates (which require MVPs with $\bfQ$) and LSQR iterates on the priorconditioned problem~\eqref{eqn:priorcondition} (which requires $\bfL_\bfQ$ or its inverse).  It is worth mentioning that this result is similar in nature to equivalence proofs for ``project-then-regularize'' and ``regularize-then-project'' approaches~\cite{HaHa93,Hansen2010}, but additional care must be taken here to handle the change of variables.

\begin{theorem}
	\label{thm:priorconequiv}
Fix $\lambda \geq 0$. Let $\bfz_k$ be the exact solution to gen-LSQR subproblem~\eqref{e_wlsqr}.  Then the $k$-th iterate of our approach, written as $\boldsymbol\mu + \bfQ\bfV_k \bfz_k$, is equivalent to $\boldsymbol\mu + \bfL_\bfQ^{-1}\bfw_k$, where $\bfw_k$ is the $k$-th iterate of LSQR on the following Tikhonov problem
\begin{equation}\label{e_standard}
\min_{\bfw}\>  \left\|\begin{pmatrix} \bfL_\bfR\bfA\bfL_\bfQ^{-1} \\ \lambda \bfI\end{pmatrix}\bfw - \begin{pmatrix} \bfL_\bfR\bfb \\ \mathbf{0}\end{pmatrix}\right\|_2^2.
\end{equation}
\end{theorem}

\begin{proof}
For notational convenience, we define $\widehat\bfA \equiv \bfL_\bfR\bfA\bfL_\bfQ^{-1}$ and $\widehat\bfb \equiv \bfL_\bfR\bfb$.  Then, LSQR applied to Tikhonov problem~\eqref{e_standard} seeks solutions $\bfw_k$ in the $k$-dimensional Krylov subspace,
\begin{align*}\hat{\mc{S}}_k  \equiv & \>  \krylov{k}(\widehat\bfA\t \widehat \bfA, \widehat \bfA\t \widehat\bfb)\\
	= & \>  \krylov{k}(\bfL_\bfQ^{-\top}\bfA^T\bfR^{-1}\bfA\bfL_\bfQ^{-1}, \bfL_\bfQ^{-\top}\bfA\t \bfR^{-1}\bfb)\\
= & \> \bfL_\bfQ^{-\top}\krylov{k}(\bfA\t\bfR^{-1}\bfA\bfQ, \bfA\t \bfR^{-1}\bfb)\\
= & \> \bfL_\bfQ^{-\top}\mc{S}_k.
\end{align*}
After $k$ steps of standard GK bidiagonalization on~\eqref{e_standard}, we get the following relation,
\begin{align*}
\widehat\bfV_{k}\t (\widehat\bfA^\top \widehat\bfA + \lambda^2 \bfI) \widehat\bfV_{k} = & \> \widehat\bfB_k^\top\widehat\bfB_k + \lambda^2 \bfI\,,
\end{align*}
and the $k$-th LSQR iterate is given by
\[\bfw_k = \widehat\bfV_{k}(\widehat\bfB_k^\top\widehat\bfB_k + \lambda^2 \bfI)^{-1} \widehat\bfV_{k}^\top \widehat\bfA^\top\widehat\bfb.\]

Following the argument in Theorem~\ref{t_lsqr_cg}, there exists an orthogonal matrix $\bfZ$~\cite[Equation 4.2.7]{meyer2000matrix} such that $\widehat\bfV_k = \bfL_\bfQ^{-\top} \bfV_k \bfZ$, and with some algebraic manipulations, we get
\[ \bfZ(\widehat\bfB_k^T\widehat\bfB_k + \lambda^2 \bfI)\bfZ\t = \bfB_k^T\bfB_k + \lambda^2 \bfI\,.\]
Finally, we can show that
\begin{align*}
\bfmu + \bfL_\bfQ^{-1}\bfw_k & = \bfmu + \bfL_\bfQ^{-1}  \widehat\bfV_{k}(\widehat\bfB_k^\top\widehat\bfB_k + \lambda^2 \bfI)^{-1} \widehat\bfV_{k}^\top \widehat\bfA^\top\widehat\bfb \\
& = \bfmu + \bfQ  \bfV_{k} \bfZ (\widehat\bfB_k^\top\widehat\bfB_k + \lambda^2 \bfI)^{-1} \bfZ\t \bfV_k\t \bfL_\bfQ^{-1} \widehat\bfA^\top\widehat\bfb \\
& = \bfmu +\bfQ\bfV_{k}(\bfB_k^\top\bfB_k + \lambda^2 \bfI)^{-1} \bfV_{k}^\top \bfQ\t \bfA^\top\bfR^{-1}\bfb \\
& = \bfmu + \bfQ\bfV_{k}\bfz_k \,,
\end{align*}
where the last step follows from the fact that $\bfV_{k}^\top \bfQ\t \bfA^\top\bfR^{-1}\bfb = \bfB_k\t \beta \bfe_1$. This concludes the proof.
\end{proof}

The significance of the above proof is that for fixed $\lambda,$ gen-LSQR iterates are equivalent to iterates of the priorconditioned approach in exact arithmetic.
Thus, as in priorconditioning, an algorithmic advantage
is that the structure of the prior is directly incorporated into the transformed operator, allowing for faster convergence under some conditions.
This typically occurs, for example, when the singular values of the prior-conditioned matrix $\widehat \bfA = \bfL_\bfR \bfA \bfL_\bfQ^{-1}$ decay faster than the singular values of $\bfA$ (c.f., Figure~\ref{fig:dpc}).  This is important because early subspaces generated by Krylov methods such as LSQR tend to contain directions corresponding to the larger singular values \cite{hansen2006deblurring}.  Hence, termination before contamination by small singular values may occur earlier.

Next we provide an illustration, similar to those for standard GK approaches \cite{GoLuOv81,saad1980rates}, where we use the heat example described in Section~\ref{s_gsvd} to show that the singular values of gen-GK bidiagonal matrix, $\bfB_k,$ and subsequently-defined tridiagonal matrix $\bar{\bfB}_k$, approximate the singular values of $\widehat \bfA$. In Figure~\ref{fig:illustration}(a), we provide the first $50$ singular values of $\widehat \bfA$, along with the singular values of $\bfB_k$ and $\bar{\bfB}_k$ for $k = 5, 20, 35, 50$.  We also provide plots of $\sigma_i^2(\bfB_k)$ to illustrate the interlacing property, where the singular values of $\bar{\bfB}_k$ interlace the squares of the singular values of $\bfB_k$ \cite{chung2015hybrid, golub1973some,thompson1976behavior,wilkinson1965algebraic,bunch1978rank}.  The impact of loss of orthogonality is evident in Figure~\ref{fig:illustration}(b), where singular value approximations $\sigma_i(\bfB_k)$ are significantly different than the desired singular values $\sigma_i(\widehat\bfA)$.

\begin{figure}[bthp]
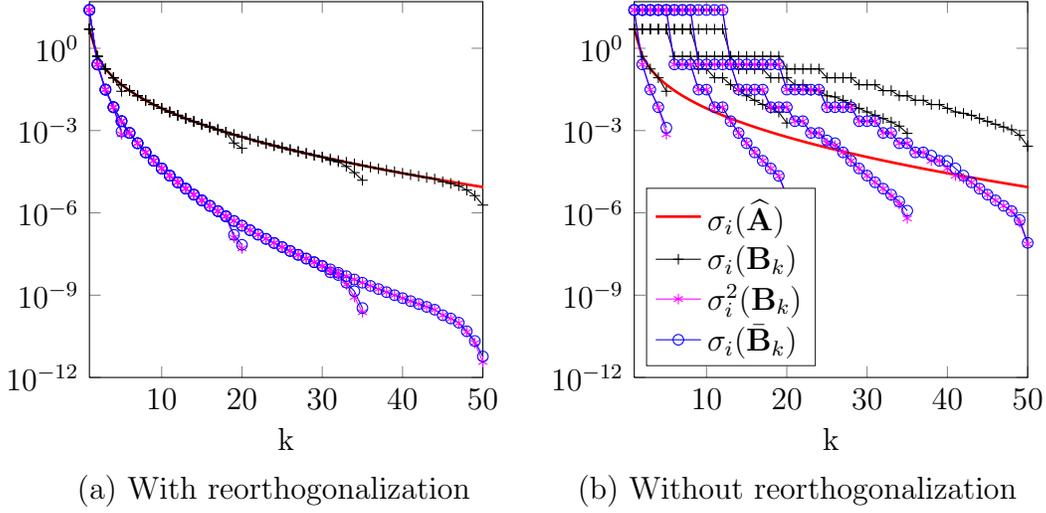

	\begin{center}
		\iwidth=55mm
		\iheight=50mm
		\begin{tabular}{cc}
			\hspace{-20pt}
			\input{figs/svdapprox2.tex} &
			\input{figs/svdapprox3.tex} \\
			(a) With reorthogonalization &(b) Without reorthogonalization
		\end{tabular}
	\end{center}
 \caption{Both figures (a) and (b) contain the first $50$ singular values of $\widehat\bfA$, along with the singular values of $\bfB_k$ and $\bar{\bfB}_k$ for $k = 5, 20, 35, 50.$  Plots of $\sigma_i^2(\bfB_k)$ are provided to illustrate an interlacing property.  Results in (a) correspond to reorthogonalization and results in (b) demonstrate the effects of loss of orthogonality.}
	\label{fig:illustration}
\end{figure}

Another significance of Theorem~\ref{thm:priorconequiv} is that the gen-GK hybrid approach can be used to incorporate prior information
without an explicit regularization term, i.e., $\lambda = 0$. Arridge et al~\cite{arridge2014iterated} further advocated priorconditioning for the additional reason that the transformed variables are dimensionless, which removes issues with physical units.  Also, for regularization functionals designed to promote edges, the authors in~\cite{arridge2014iterated} noted that Krylov subspace methods have poor convergence whereas the transformed problem amplifies directions spanning the prior covariance, thereby improving convergence. They modeled the prior precision matrix $\bfQ^{-1}$ as a sparse matrix obtained from the discretization of a PDE operator and proposed a factorization-free preconditioned LSQR approach called MLSQR that involves generating a basis for the Krylov subspace \begin{equation}
\label{eq:Arridge}
	\krylov{k}(\bfQ\bfA^\top\bfR^{-1}\bfA,\bfQ\bfA^\top\bfR^{-1}\bfb).
\end{equation}

The computational advantage of gen-LSQR over priorconditioning and MLSQR is that we only require access to $\bfQ$ via MVPs. On the contrary, priorconditioning explicitly requires matrix $\bfL_\bfQ^{-1}$(i.e., the square-root of $\bfQ$) or solves with $\bfL_\bfQ$, both of which can be prohibitively expensive. In MLSQR, generating a basis for~\eqref{eq:Arridge} requires repeated MVPs with $\bfQ$, which amounts to several solves using $\bfQ^{-1}$ (recall that~\cite{arridge2014iterated} assumes that $\bfQ^{-1}$ is sparse and forming $\bfQ$ is unnecessary). For the priors of interest here, entries of $\bfQ$ are modeled directly and MVPs with $\bfQ$ can be done efficiently, so the proposed gen-LSQR and corresponding generalized hybrid approach are computationally advantageous over priorconditioning and MLSQR.

\section{Numerical Experiments} % (fold)
\label{sec:numerical_experiments}

In this section, we provide two examples from image reconstruction -- the first is a model problem in seismic traveltime tomography, and the second is a synthetic problem in super-resolution image reconstruction.  In all of the provided results, gen-LSQR and gen-HyBR correspond to solutions computed as~\eqref{eqn:undo_change} where $\bfx_k$ is the solution to~\eqref{e_wlsqr}, where $\lambda=0$ for gen-LSQR and $\lambda$ for gen-HyBR was selected using the various methods described in Section~\ref{s_param}.

\subsection{Seismic tomography reconstruction} % (fold)
\label{sub:experiment_1_seismic_tomography}

For our first application, we consider a synthetic test problem from cross-well seismic tomography~\cite{ambikasaran2013large}. The goal is to image the slowness in the medium, where slowness is defined as the reciprocal of seismic velocity. Seismic sources are fired sequentially at each of the sources and the time delay between firing and recording is measured at the receivers; this is the travel time. As a first order approximation, the seismic wave is assumed to travel along a straight line from the sources to the receivers without reflections or refractions. One measurement is recorded for each source-receiver pair. All together we have $n_\text{sou} = 50$ sources and $n_\text{rec} = 50$ receivers. and therefore, there are $m = n_\text{rec} n_\text{sou}$ measurements. The domain is discretized into $\sqrt{n} \times \sqrt{n}$ cells and within each cell, the slowness is assumed to be constant. Therefore, the travel time is a sum of the slowness in the cell, weighted by the length of the ray within the cell. The inverse problem is, therefore, to reconstruct the slowness of the medium from the discrete measurements of the travel times. We assume that the travel times are corrupted by Gaussian noise, so that the measurement takes the form in~\eqref{eq:problem}
where $\bfd$ are the observed (synthetic) travel times, $\bfs$ is the slowness that we are interested in imaging and $\bfA$ is the measurement operator, whose rows correspond to each source-receiver pair and are constructed such that their inner product with the slowness would result in the travel time. As constructed above, each row of $\bfA$ has $\mc{O}(m\sqrt{n})$ entries, so that  $\bfA$ is a sparse matrix with $\mc{O} (m\sqrt{n})$ non-zero entries. As the ``true field,'' we use a truncated Karhunen-Lo\'{e}ve expansion
$$s(\bfx)\> = \> \mu(\bfx) + \sum_{k=1}^{N_k} \sqrt{\lambda_k}\xi_k {\phi}_k(\bfx),$$
where $\xi_k \sim \mc{N}(0,1)$ are i.i.d.\ random variables, and $(\lambda_k,\phi_k)$ are the eigenpairs of the integral operator

$$T: L^2(\Omega)\rightarrow L^2(\Omega), \qquad (T\phi)(\bfx) \equiv \int_\Omega \kappa(\bfx,\bfy)\phi(\bfy)d\bfy,$$
and covariance kernel
$\kappa(\bfx,\bfy) = \theta\exp\left( -(r/L)^2 \right)$. Furthermore, we choose $\theta = 1\times 10^{-3}$, $\mu(\bfx) = 0.08$ seconds/meter and $L = 100 $ meters. Here $N_k$ is chosen to be $20$.

\begin{figure}[!ht]\centering
	\label{fig:RTrelerror}
\includegraphics[scale=.5]{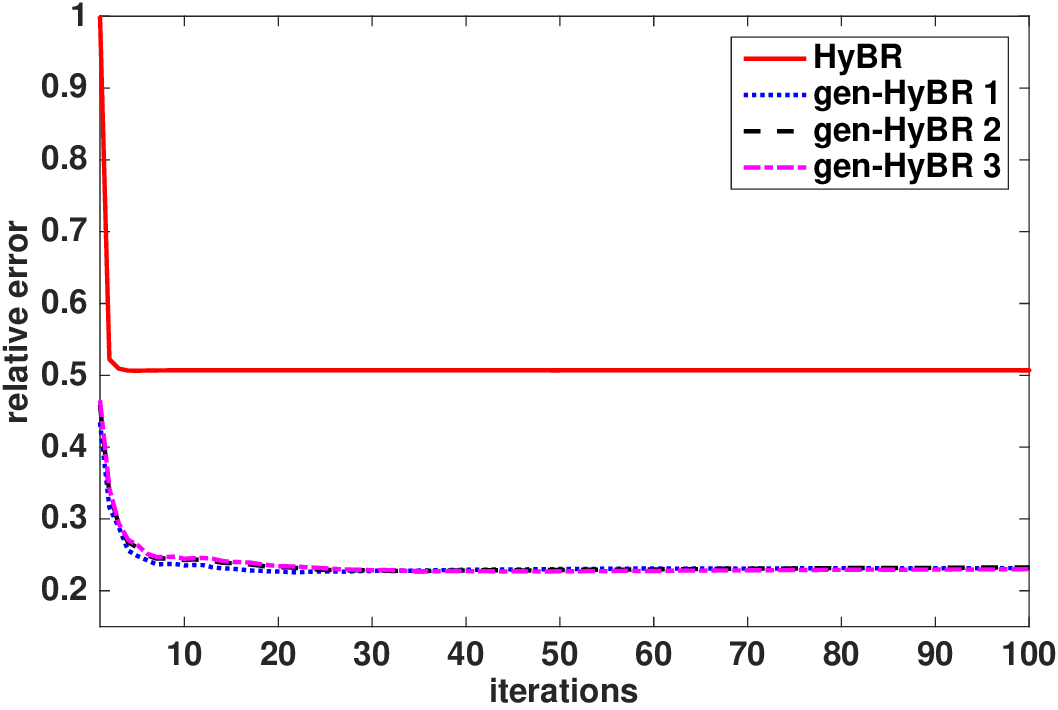}
\caption{Comparison of standard and generalized iterative methods for ray-tomography application.  Results for HyBR and gen-HyBR correspond to hybrid methods, where the optimal regularization parameter was used at each iteration. Here `gen-HyBR 1' corresponds to $\nu = 1/2$, `gen-HyBR 2' corresponds to $\nu = 3/2$, and  `gen-HyBR 3' corresponds to $\nu = 5/2$, where $\nu$ is a parameter of the Mat\`{e}rn kernel. }
\end{figure}

To avoid an ``inverse crime'', we use a different covariance kernel for the reconstruction. We consider three kernels from the Mat\'{e}rn covariance family, with the parameters $\nu = 1/2,3/2,5/2$ and $\alpha = L^{-1}$. Furthermore, we added $2\%$ Gaussian noise to simulate measurement error, i.e., $\frac{\norm{\bfepsilon}{2}}{\norm{\bfA\bfs_\true}{2}}=.02$.
In this application, we study the effect of the covariance kernels on the reconstruction; therefore, for the reconstructions, optimal regularization parameters were used. A numerical comparison of various parameter selection techniques will be provided in the next example.

Figure~\ref{fig:RTrelerror} shows the iteration history of the relative error.  `HyBR' corresponds to  $\bfQ = \bfI$ solved using the standard LSQR approach, whereas `gen-HyBR 1/2/3' correspond to $\bfQ$ constructed using a  Mat\'{e}rn kernel with $\nu=1/2,3/2,5/2$ respectively. It is evident that errors corresponding to the generalized hybrid methods are much lower than those of the standard approach. This is also confirmed in the reconstructions provided in Figure~\ref{fig:RTreconimages_mat}.  Reconstructions for gen-HyBR 1 and gen-HyBR 3 were similar to gen-HyBR 2 and are not presented.

\begin{figure}[!ht]\centering
		\label{fig:RTreconimages_mat}
\includegraphics[scale=0.3]{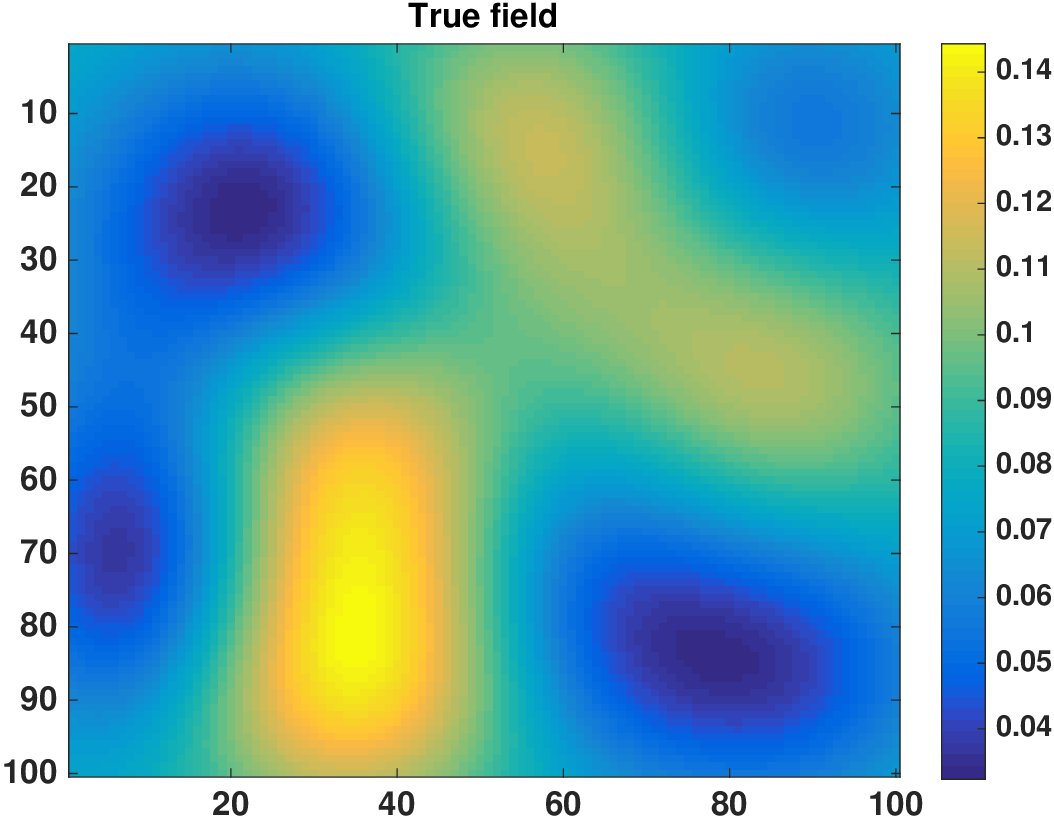}
\includegraphics[scale=0.3]{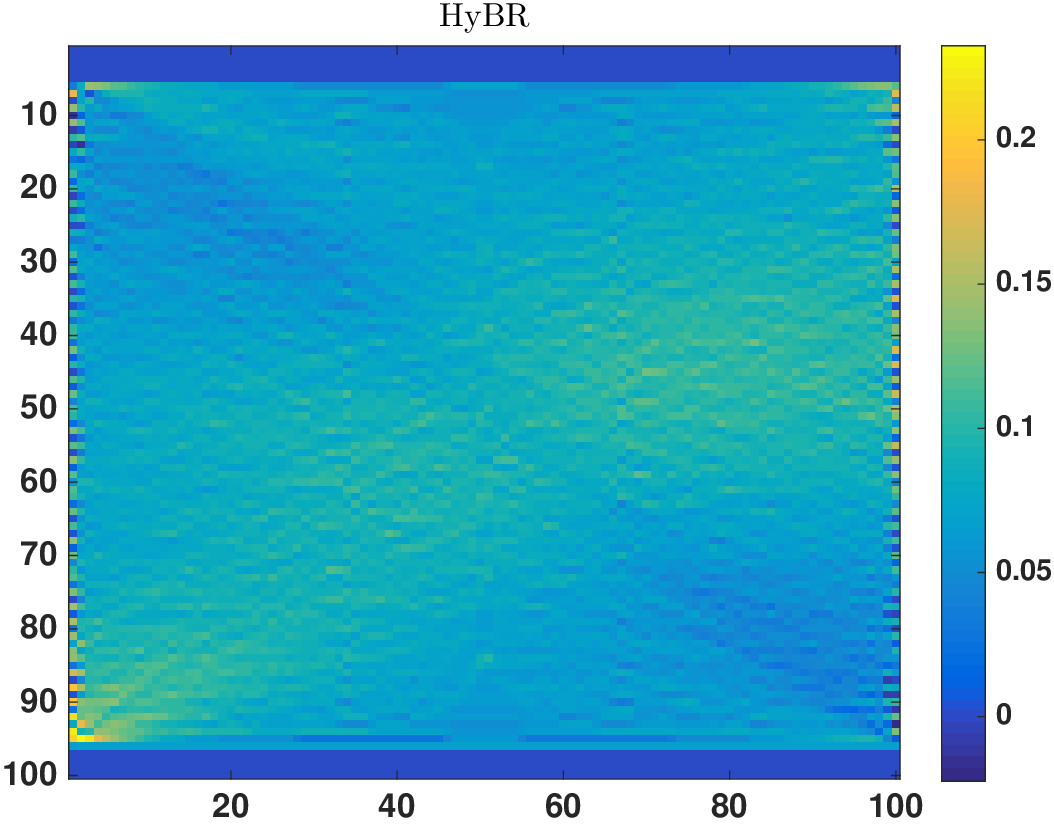}
\includegraphics[scale=0.3]{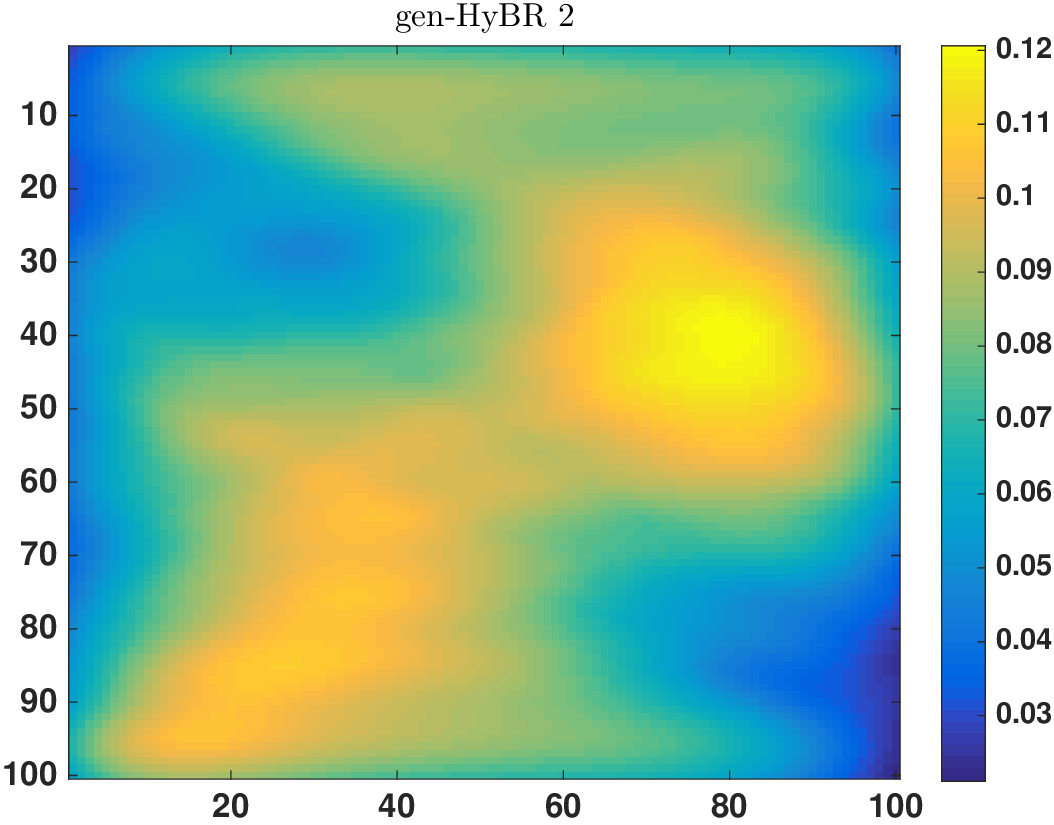}
\caption{(top left) True image, (top right) reconstruction using HyBR, optimal $\lambda^2 = 4.71\times 10^{-1}$ (bottom) reconstruction using gen-HyBR (Mat\`{e}rn kernel $\nu=3/2$), optimal  $\lambda^2 = 2.32\times 10^{-7}$.}
\end{figure}

\subsection{Super-resolution image reconstruction} % (fold)
\label{sub:experiment_2_super_resolution_imaging}
In this experiment, we consider a super-resolution imaging example, where the goal is to construct one high-resolution image from a set of low-resolution ones.   Super-resolution imaging has gained increasing popularity over the last few years, as it allows scientists to obtain images with higher spatial resolution, without having to sacrifice signal-to-noise ratio and dynamic range \cite{park2003super,farsiu2004advances,ChHaNa06}.  Let $\bfb_i, i=1,\dots,K$ represent the set of vectorized low-resolution images, each containing different information of the same scene. Then super-resolution imaging can be modeled as a linear inverse problem~\eqref{eq:problem}, where $\bfs$ represents a vectorized high-resolution image and
	\begin{equation}
	\bfA = \begin{bmatrix} \bfA_1 \\ \vdots \\ \bfA_K
	\end{bmatrix} \quad \mbox{and}\quad \bfb = \begin{bmatrix} \bfb_1 \\ \vdots \\ \bfb_K
	\end{bmatrix}
\end{equation}
where $\bfA_i$ models the forward process corresponding to $\bfb_i$.  More specifically, each $\bfA_i$ represents a linear transformation (e.g., rotation or linear shift) of the high-resolution image, followed by a restriction operation that takes a high-resolution image to a low-resolution image.  Bilinear interpolation was used to generate the sparse interpolation matrix in $\bfA_i$ and a Kronecker product was used to represent the restriction operation, as described in \cite{ChHaNa06}.
For this example, we assume that all parameters defining $\bfA$ (e.g., rotation and translation parameters) are known.  If this is not the case, then such knowledge may be estimated using image registration techniques \cite{modersitzki2003numerical}, or a separable nonlinear LS framework can be considered for simultaneous estimation of both the parameters and the high-resolution image \cite{ChHaNa06, Chung2010b}.  In this paper, we investigate generalized iterative approaches for solving the linear super-resolution problem, but we remark that this approach could also be used to solve the linear subproblem within a nonlinear optimization scheme.  Thus, for this example, the problem can be stated as given $\bfA$ and $\bfb$, the goal of super-resolution imaging is to reconstruct the true high-resolution image.

For this problem, the high-resolution image of a MRI brain slice consisted of $128\times 128$ pixels and is shown in Figure~\ref{fig:SRexample}(a). We generated 5 low-resolution images of size $32 \times 32$ pixels, where each image is slightly rotated from the others.  White noise was added to the low-resolution images, where the noise level was $.02$.  Four of the observed low-resolution images are shown in Figure~\ref{fig:SRexample}(b). In all of the results, we assumed $\bfR = \bfI$ and $\bfmu=\bfzero$.
\begin{figure}[!bt]
 	\label{fig:SRexample}
 	\begin{center}
		\begin{tabular}{ccc}
\multirow{2}{*}[6em]{\includegraphics[scale=.4]{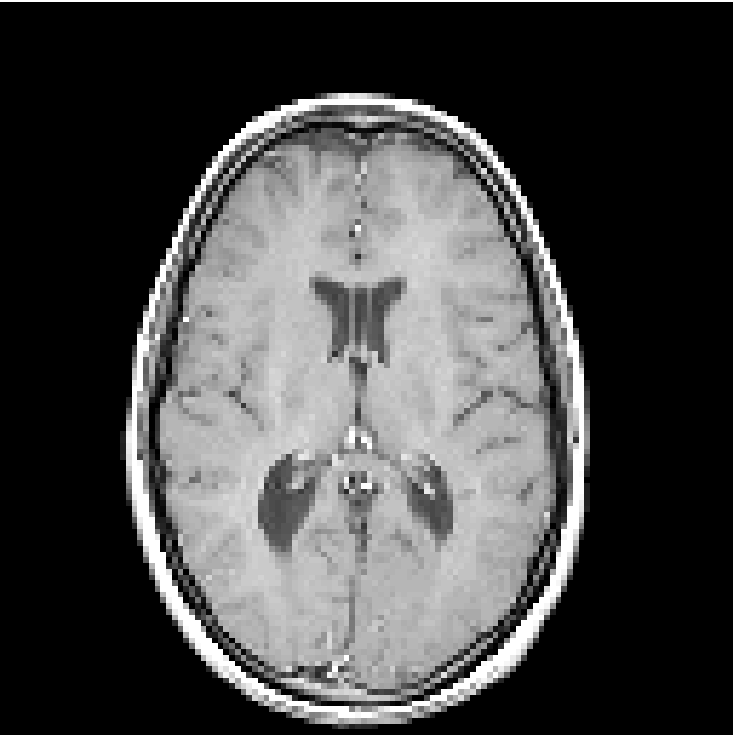}} &
		\includegraphics[scale=.2]{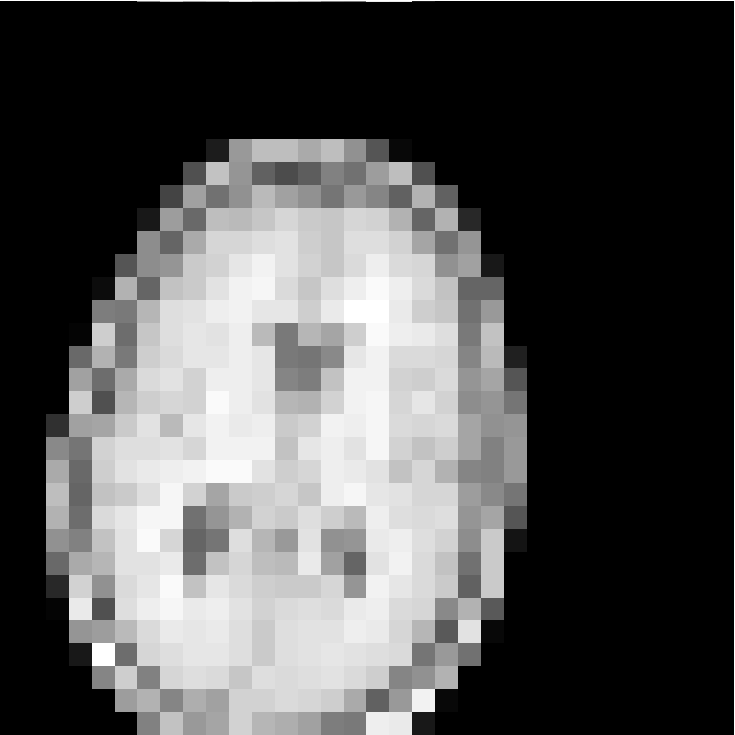}  &
		\includegraphics[scale=.2]{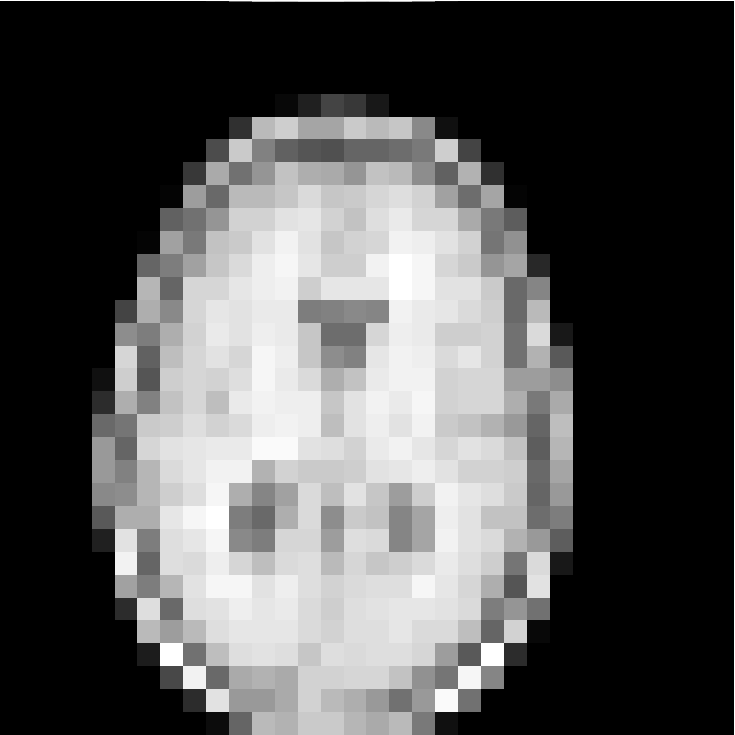}  \\
&		\includegraphics[scale=.2]{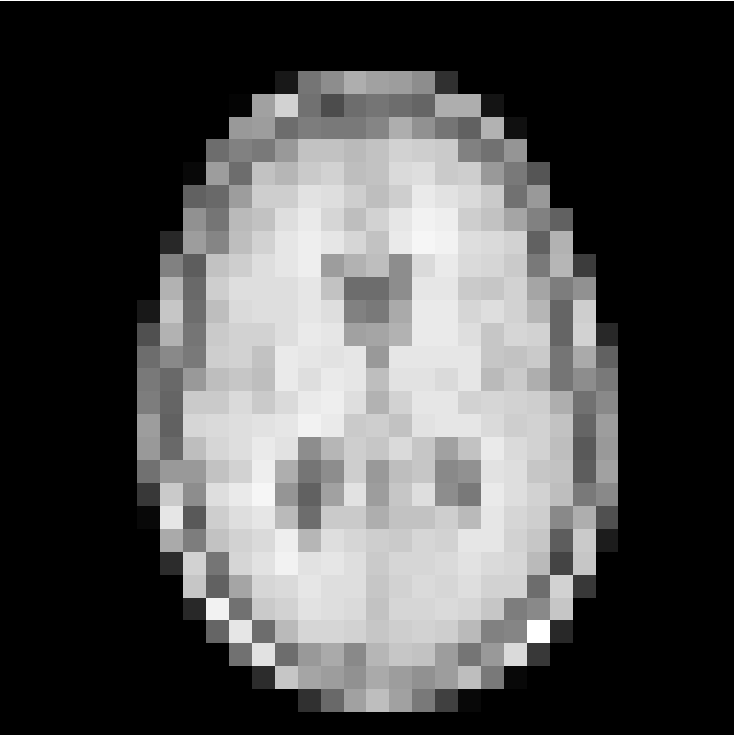}  &
		\includegraphics[scale=.2]{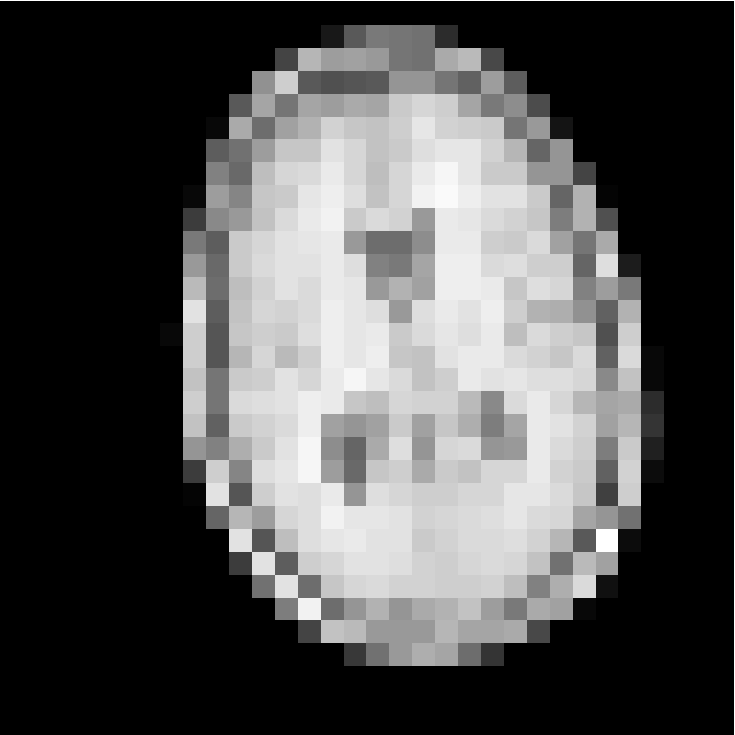} \\
		(a) High resolution image & \multicolumn{2}{c}{(b) Sample low-resolution images} 		\end{tabular}
 	\end{center}
 \caption{Super-resolution imaging example.}
 \end{figure}

First, we compare standard iterative methods where $\bfQ = \bfI$ to the proposed generalized LSQR and LSMR methods and their generalized hybrid variants, where $\bfQ$ represents a Mat\'{e}rn kernel with $\nu=0.5$ and $\alpha = 0.007$.  In Figure~\ref{fig:SRrelerror}, we provide the relative error plots, which show that the generalized approaches can produce reconstructions with smaller relative reconstruction error than the standard approaches, and with an appropriate choice of the regularization parameter (here we use $\lambda_{\rm opt}$), semi-convergence can be avoided.  Methods based on LSMR exhibit delayed semiconvergence, as observed and discussed in \cite{chung2015hybrid}.  Subimages of the HyBR-opt and gen-HyBR-opt reconstructions at iteration 100 are shown in Figure~\ref{fig:SRreconimages}, where it is evident that using the generalized approach can result in improved, smoother image reconstructions.

\begin{figure}[!bt]\centering
	\label{fig:SRrelerror}
\includegraphics[scale=.5]{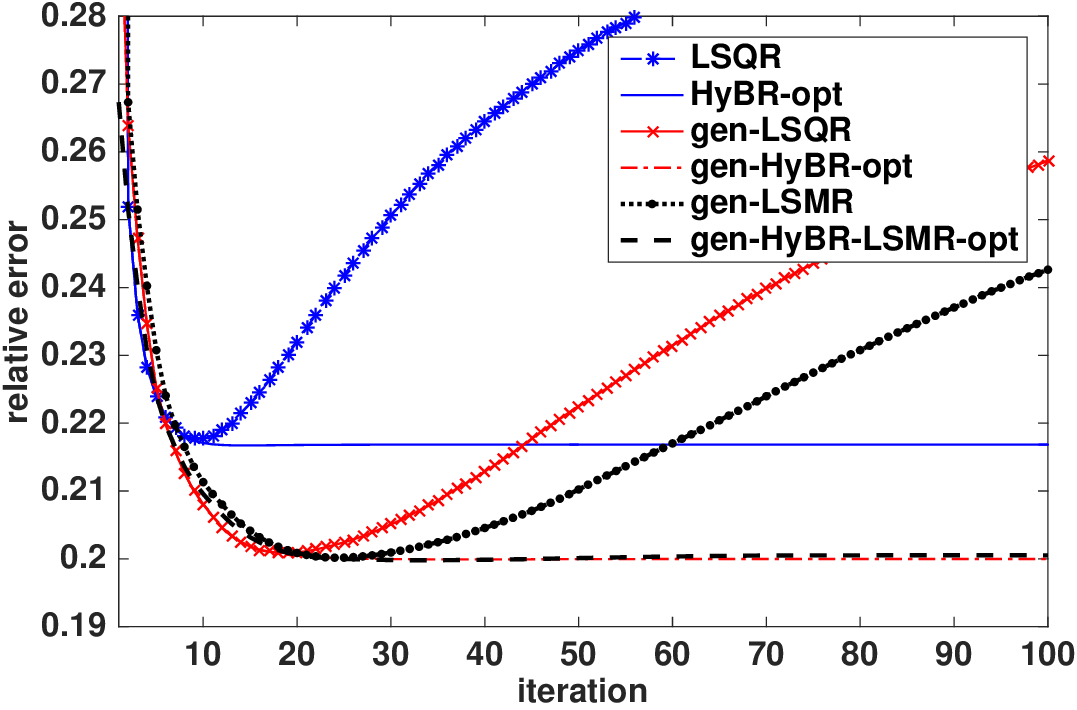}
\caption{Comparison of standard and generalized iterative methods for super-resolution image reconstruction.  Results for HyBR-opt, gen-HyBR-opt, and gen-HyBR-LSMR-opt correspond to hybrid methods, where the optimal regularization parameter was used at each iteration.}
\end{figure}

\begin{figure}[!bt]\centering
		\label{fig:SRreconimages}
\includegraphics[scale=.75]{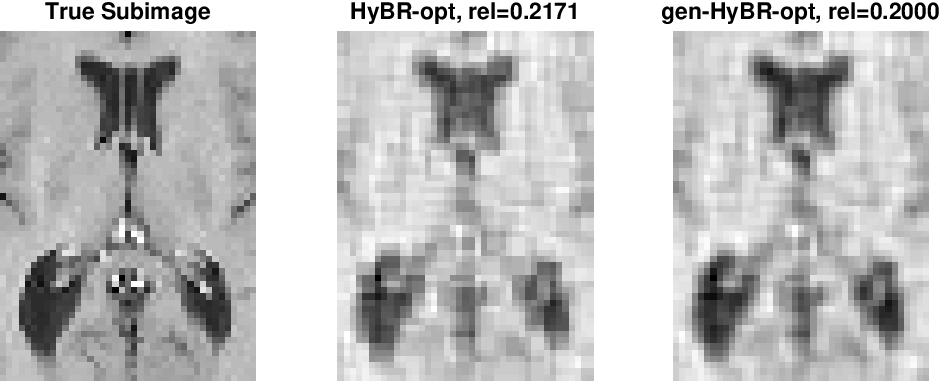}
\caption{Subimages of reconstructed images for HyBR-opt and gen-HyBR-opt demonstrate a qualitative improvement of generalized hybrid methods, where $\bfQ$ represents a covariance matrix from the Mat\'{e}rn class.  The subimage of the true image is provided on the left for reference and relative reconstruction errors are provided as rel.}
\end{figure}

In Figure~\ref{fig:SRrelerror_reg}, we provide relative reconstruction errors for gen-HyBR, where different methods were used to select the regularization parameter.  In addition to gen-LSQR (where $\lambda = 0$) and gen-HyBR-opt, which are both in Figure~\ref{fig:SRrelerror}, we also consider parameters described in Section~\ref{s_param}, namely, $\lambda_{\rm dp},\lambda_{\rm gcv}$, and $\lambda_{\rm wgcv}$. For the discrepancy principle, we used the true value for
$\delta$
% $\epsilon$
and $\tau=1$ in~\eqref{e_dp}.  Such values may not be available in practice.  We refer the interested reader to other works on noise estimation, e.g., \cite{donoho1995noising}.  Even with the actual value of $\nu$ in~\eqref{e_upre_proj}, UPRE resulted in poor reconstructions in our experience, so we do not include those results here.

\begin{figure}[!ht]\centering
	\label{fig:SRrelerror_reg}
\includegraphics[scale=.5]{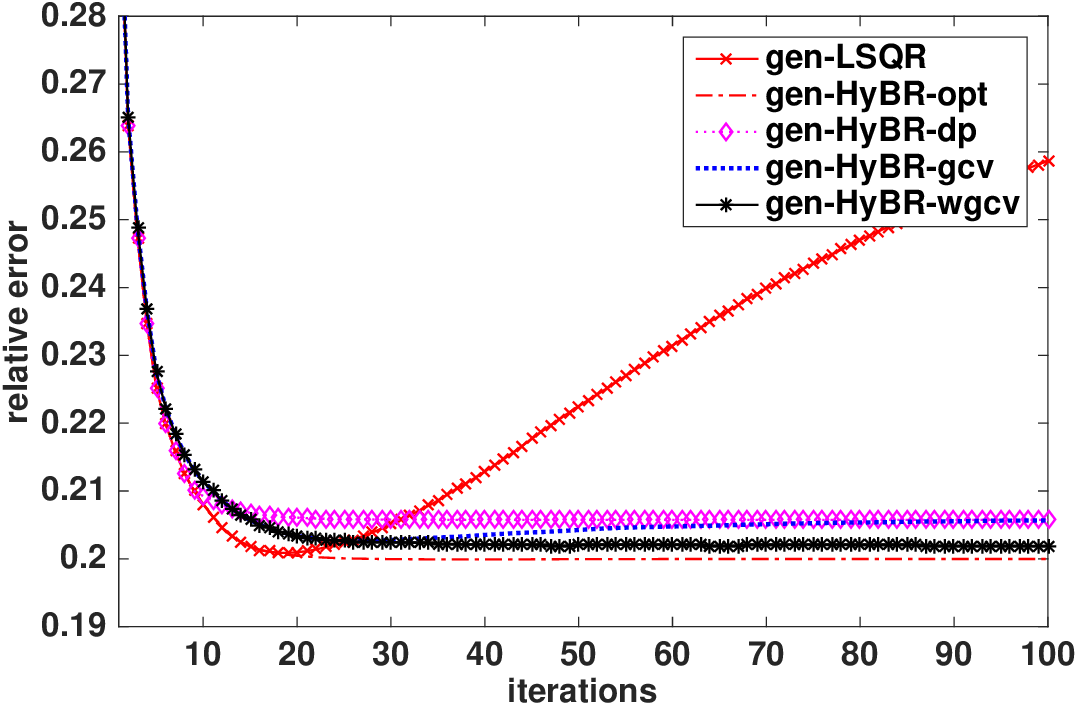}
\caption{Relative error curves for generalized hybrid methods, for various regularization parameter choice methods. gen-LSQR uses $\lambda = 0,$ gen-HyBR-opt uses the optimal regularization parameter (that which minimizes the reconstruction error), dp represents the discrepancy principle, gcv represents the generalized cross validation method, and wgcv represents a weighted variant of gcv.}
\end{figure}

In the next experiment, we consider the impact of different choices of $\bfQ$.  We choose $\bfQ_1,$ $\bfQ_2,$ and $\bfQ_3,$ from the Mat\'{e}rn covariance family where $\nu = 1/2,1/2,\infty$ and $\alpha = .007, .003, .007$ respectively.  For all reconstructions, we used $\lambda_{\rm wgcv}$.  Relative reconstruction  errors and reconstructed images are provided in Figures~\ref{fig:SRmatern} and~\ref{fig:SRreconimages_mat} respectively, where absolute error images are provided in inverted colormap, so that black corresponds to larger absolute reconstruction error.
\begin{figure}[bthp!]\centering
		\label{fig:SRmatern}
\includegraphics[scale=.5]{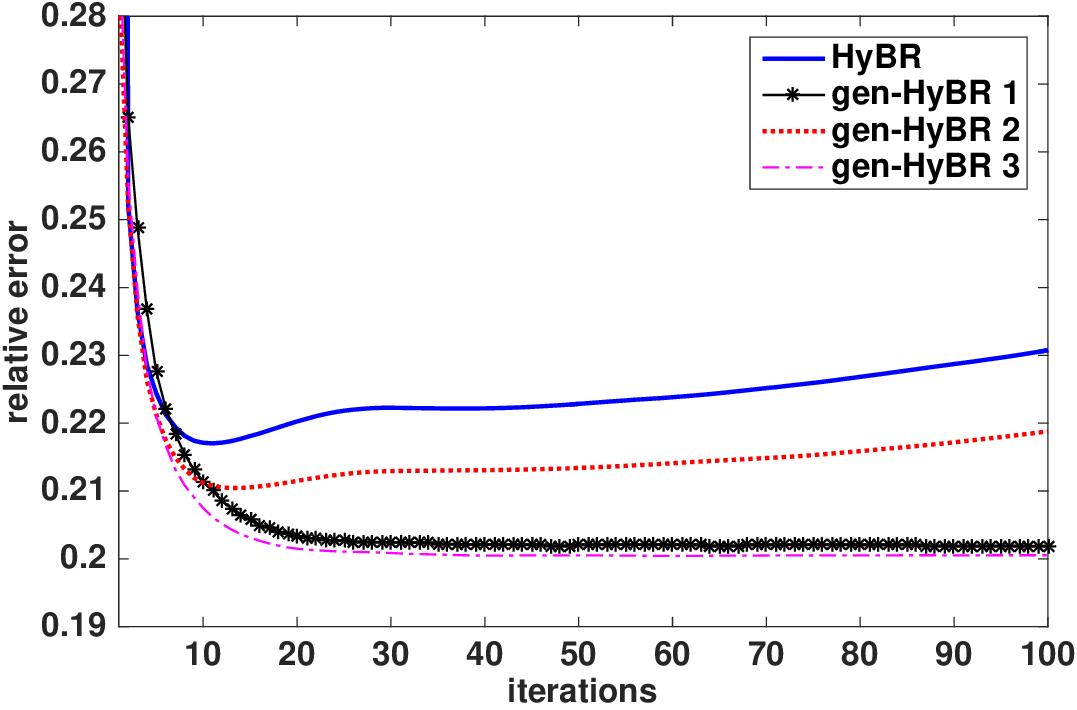}
\caption{Relative error curves for different covariance matrices from the Mat\'{e}rn class for the super-resolution imaging example.  All results used WGCV to select regularization parameters.}
\end{figure}

\begin{figure}[bthp!]\centering
		\label{fig:SRreconimages_mat}
\includegraphics[scale=.8]{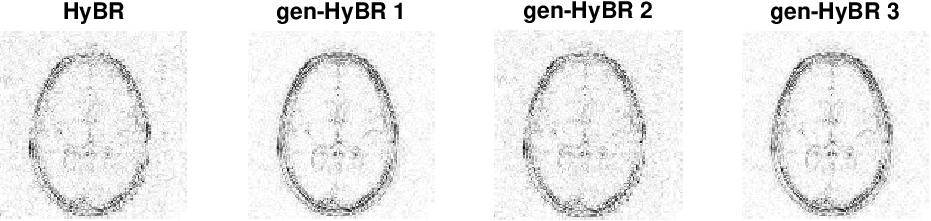}
\caption{Absolute error images (in inverted colormap, where white corresponds to 0) for reconstructions obtained via HyBR($\bfQ = \bfI$) and gen-HyBR for various $\bfQ$ from the  Mat\'{e}rn covariance family.  All results correspond to regularization parameters computed using WGCV.}
\end{figure}
% subsection experiment_2_super_resolution_imaging (end)

% section numerical_experiments (end)

\section{Conclusions} % (fold)
\label{sec:conclusions}
In this paper, we described a generalized hybrid iterative method for computing maximum a posteriori estimators for large-scale Bayesian inverse problems, also known as general-form Tikhonov solutions. Our approach is different from previous approaches since we avoid forming (either explicitly, or MVPs with) the square root or inverse of the prior covariance matrix.  By exploiting the shift-invariant property of Krylov subspaces generated by the generalized Golub-Kahan bidiagonalization process, the proposed hybrid approach has the benefit that regularization parameters can be determined automatically using the projected problem.  Theoretical results provide connections to GSVD filtered solutions, priorconditioned solutions, and iterative methods with weighted inner products.  Numerical results from seismic tomography reconstruction and super-resolution imaging validate the effectiveness and efficiency of our approach.

\section{Acknowledgements}
Some of this work was conducted as a part of SAMSI Program on Computational Challenges in Neuroscience (CCNS) 2015-2016. This material was based upon work partially supported by the National Science Foundation under Grant DMS--1127914 to the Statistical and Applied Mathematical Sciences Institute.

\appendix
\section{Parameter selection methods} % (fold)
\label{sec:appendix}
Parameter selection using the GCV was described in Section~\ref{s_param}. In this appendix, we give details of other parameter selection techniques such as
the Discrepancy Principle (DP) and the Unbiased Predictive Risk Estimator (UPRE).  These methods rely on an \textit{a priori} estimate for the noise level. For DP, the regularization parameter, $\lambda_{\rm dp},$ is selected so that the residual norm is on the order of the noise in the data, that is
\begin{equation}
	\label{e_dp}
\norm{\bfA \bfs_\lambda- \bfd}{\bfR^{-1}}^2 = \tau \delta,
\end{equation}
where $\delta$ is an approximation to the expected value of the squared norm of the noise and $\tau \gtrsim 1$ is a user-defined parameter.
From our assumption of the noise $\B{\epsilon}$, the expression $\bfL_\bfR\B{\epsilon}$ is a whitening transformation and it follows that the expected value of $\norm{\B{\epsilon}}{\bfR^{-1}}^2 = \normtwo{\bfL_\bfR\B{\epsilon}}^2 $ is approximately $m$.
In the generalized hybrid approach, we seek the DP parameter, $\lambda_{\rm dp},$ that satisfies $\normtwo{\beta_1\bfe_1 - \bfB_k\bfz_{k,\lambda}}^2 = \tau\delta$ at each iteration $k$, where $\delta \approx  m$. The UPRE method selects a regularization parameter that minimizes an unbiased estimator for the expected value of the predictive risk. More concretely, UPRE parameter, $\lambda_{\rm upre},$ minimizes
\begin{equation}\label{e_upre}
 % U(\lambda) \equiv \frac{1}{n}\norm{\bfA\bfs_\lambda - \bfd}{\bfR^{-1}}^2 + \frac{2\eta^2}{n}\trace\left(\bfL_\bfR \bfA\bfA^\dagger_\lambda\right) - \eta^2\,,
 U(\lambda) \equiv \frac{1}{m}\norm{\bfA\bfs_\lambda - \bfd}{\bfR^{-1}}^2 + \frac{2}{m}\trace\left(\bfL_\bfR \bfA\bfA^\dagger_\lambda\right) - 1\,,
\end{equation}
where again the GSVD, if available, could be used to simplify the expression. At each generalized hybrid iteration, regularization parameter, $\lambda_{\rm upre},$ is chosen to minimize the UPRE function corresponding to the projected problem,
\begin{equation}\label{e_upre_proj}
% U_\text{proj}(\lambda) \equiv \frac{1}{k}\normtwo{\bfB_k\bfz_{k,\lambda}- \beta_1\bfe_1}^2 + \frac{2\eta^2}{k}\trace\left(\bfB_k\bfB_{k,\lambda}^\dagger\right) - \eta^2.
U_\text{proj}(\lambda) \equiv \frac{1}{k}\normtwo{\bfB_k\bfz_{k,\lambda}- \beta_1\bfe_1}^2 + \frac{2}{k}\trace\left(\bfB_k\bfB_{k,\lambda}^\dagger\right) - 1.
\end{equation}
An analysis in \cite{renaut2015hybrid} showed that the obtained UPRE regularization parameter for the projected problem (i.e., the minimizer of~\eqref{e_upre_proj}) provides a good approximation for the regularization parameter for the full problem on the projected space.

\bibliography{references}

\begin{thebibliography}{10}

\bibitem{ambikasaran2012large}
{\sc S.~Ambikasaran, J.~Li, P.~Kitanidis, and E.~Darve}, {\em Large-scale
  stochastic linear inversion using {H}ierarchical matrices}, Computational
  Geosciences,  (2012).

\bibitem{ambikasaran2013large}
{\sc S.~Ambikasaran, J.~Y. Li, P.~K. Kitanidis, and E.~Darve}, {\em Large-scale
  stochastic linear inversion using hierarchical matrices}, Computational
  Geosciences, 17 (2013), pp.~913--927.

\bibitem{arioli2013generalized}
{\sc M.~Arioli}, {\em Generalized {G}olub--{K}ahan bidiagonalization and
  stopping criteria}, SIAM Journal on Matrix Analysis and Applications, 34
  (2013), pp.~571--592.

\bibitem{arridge2014iterated}
{\sc S.~Arridge, M.~Betcke, and L.~Harhanen}, {\em Iterated preconditioned
  {LSQR} method for inverse problems on unstructured grids}, Inverse Problems,
  30 (2014), p.~075009.

\bibitem{Bazan2010}
{\sc F.~Baz{\'a}n and L.~Borges}, {\em {GKB-FP}: an algorithm for large-scale
  discrete ill-posed problems}, BIT Numerical Mathematics, 50 (2010),
  pp.~481--507.

\bibitem{bunch1978rank}
{\sc J.~R. Bunch, C.~P. Nielsen, and D.~C. Sorensen}, {\em Rank-one
  modification of the symmetric eigenproblem}, Numerische Mathematik, 31
  (1978), pp.~31--48.

\bibitem{calvetti2007preconditioned}
{\sc D.~Calvetti}, {\em Preconditioned iterative methods for linear discrete
  ill-posed problems from a {B}ayesian inversion perspective}, Journal of
  Computational and Applied Mathematics, 198 (2007), pp.~378--395.

\bibitem{calvetti2005priorconditioners}
{\sc D.~Calvetti and E.~Somersalo}, {\em Priorconditioners for linear systems},
  Inverse problems, 21 (2005), pp.~1397--1418.

\bibitem{calvetti2007introduction}
{\sc D.~Calvetti and E.~Somersalo}, {\em An Introduction to {B}ayesian
  Scientific Computing: Ten Lectures on Subjective Computing}, vol.~2,
  Springer, New York, 2007.

\bibitem{ChHaNa06}
{\sc J.~Chung, E.~Haber, and J.~G. Nagy}, {\em Numerical methods for coupled
  super-resolution}, Inverse Problems, 22 (2006), pp.~1261--1272.

\bibitem{Chung2010b}
{\sc J.~Chung and J.~G. Nagy}, {\em An efficient iterative approach for
  large-scale separable nonlinear inverse problems}, SIAM Journal on Scientific
  Computing, 31 (2010), pp.~4654--4674.

\bibitem{ChNaOLe08}
{\sc J.~Chung, J.~G. Nagy, and D.~P. O'Leary}, {\em A weighted {GCV} method for
  {L}anczos hybrid regularization}, Electronic Transactions on Numerical
  Analysis, 28 (2008), pp.~149--167.

\bibitem{chung2015hybrid}
{\sc J.~Chung and K.~Palmer}, {\em A hybrid {LSMR} algorithm for large-scale
  {T}ikhonov regularization}, SIAM Journal on Scientific Computing, 37 (2015),
  pp.~S562--S580.

\bibitem{donoho1995noising}
{\sc D.~L. Donoho}, {\em De-noising by soft-thresholding}, IEEE Transactions on
  Information Theory, 41 (1995), pp.~613--627.

\bibitem{farsiu2004advances}
{\sc S.~Farsiu, D.~Robinson, M.~Elad, and P.~Milanfar}, {\em Advances and
  challenges in super-resolution}, International Journal of Imaging Systems and
  Technology, 14 (2004), pp.~47--57.

\bibitem{fong2011lsmr}
{\sc D.~C. Fong and M.~Saunders}, {\em {LSMR}: An iterative algorithm for
  sparse least-squares problems}, SIAM J. Scientific Computing, 33 (2011),
  pp.~2950--2971.

\bibitem{gazzola2014generalized}
{\sc S.~Gazzola and J.~G. Nagy}, {\em Generalized {A}rnoldi--{T}ikhonov method
  for sparse reconstruction}, SIAM Journal on Scientific Computing, 36 (2014),
  pp.~B225--B247.

\bibitem{gazzola2015krylov}
{\sc S.~Gazzola, P.~Novati, and M.~R. Russo}, {\em On {K}rylov projection
  methods and {T}ikhonov regularization}, Electron. Trans. Numer. Anal, 44
  (2015), pp.~83--123.

\bibitem{golub1973some}
{\sc G.~H. Golub}, {\em Some modified matrix eigenvalue problems}, SIAM Review,
  15 (1973), pp.~318--334.

\bibitem{GoHeWa79}
{\sc G.~H. Golub, M.~Heath, and G.~Wahba}, {\em Generalized cross-validation as
  a method for choosing a good ridge parameter}, Technometrics, 21 (1979),
  pp.~215--223.

\bibitem{GoKa65}
{\sc G.~H. Golub and W.~Kahan}, {\em Calculating the singular values and
  pseudoinverse of a matrix}, SIAM Journal on Numerical Analysis, 2 (1965),
  pp.~205--224.

\bibitem{GoLuOv81}
{\sc G.~H. Golub, F.~T. Luk, and M.~L. Overton}, {\em A block {L}anczos method
  for computing the singular values and corresponding singular vectors of a
  matrix}, ACM Trans. Math Softw., 7 (1981), pp.~149--169.

\bibitem{HaHa93}
{\sc M.~Hanke and P.~C. Hansen}, {\em Regularization methods for large-scale
  problems}, Surveys on Mathematics for Industry, 3 (1993), pp.~253--315.

\bibitem{hansen1994regularization}
{\sc P.~C. Hansen}, {\em Regularization tools: A {M}atlab package for analysis
  and solution of discrete ill-posed problems}, Numerical algorithms, 6 (1994),
  pp.~1--35.

\bibitem{Hansen2010}
{\sc P.~C. Hansen}, {\em Discrete Inverse Problems: Insight and Algorithms},
  SIAM, Philadelphia, 2010.

\bibitem{hansen2006deblurring}
{\sc P.~C. Hansen, J.~G. Nagy, and D.~P. O'leary}, {\em Deblurring images:
  matrices, spectra, and filtering}, vol.~3, SIAM, 2006.

\bibitem{hnvetynkova2009regularizing}
{\sc I.~Hn{\v{e}}tynkov{\'a}, M.~Ple{\v{s}}inger, and Z.~Strako{\v{s}}}, {\em
  The regularizing effect of the {G}olub--{K}ahan iterative bidiagonalization
  and revealing the noise level in the data}, BIT Numerical Mathematics, 49
  (2009), pp.~669--696.

\bibitem{hochstenbach2010iterative}
{\sc M.~E. Hochstenbach and L.~Reichel}, {\em An iterative method for
  {T}ikhonov regularization with a general linear regularization operator}, J.
  Integral Equations Appl, 22 (2010), pp.~463--480.

\bibitem{hochstenbach2015golub}
{\sc M.~E. Hochstenbach, L.~Reichel, and X.~Yu}, {\em A {G}olub--{K}ahan-type
  reduction method for matrix pairs}, Journal of Scientific Computing, 65
  (2015), pp.~767--789.

\bibitem{horn2012matrix}
{\sc R.~A. Horn and C.~R. Johnson}, {\em Matrix analysis}, Cambridge university
  press, 2012.

\bibitem{jensen2007iterative}
{\sc T.~K. Jensen and P.~C. Hansen}, {\em Iterative regularization with
  minimum-residual methods}, BIT Numerical Mathematics, 47 (2007),
  pp.~103--120.

\bibitem{kilmer2007projection}
{\sc M.~E. Kilmer, P.~C. Hansen, and M.~I. Espa\~{n}ol}, {\em A
  projection-based approach to general-form {T}ikhonov regularization}, SIAM
  Journal on Scientific Computing, 29 (2007), pp.~315--330.

\bibitem{Kilmer2001}
{\sc M.~E. Kilmer and D.~P. O'Leary}, {\em Choosing regularization parameters
  in iterative methods for ill-posed problems}, SIAM Journal on Matrix Analysis
  and Applications, 22 (2001), pp.~1204--1221.

\bibitem{lindgren2011explicit}
{\sc F.~Lindgren, H.~Rue, and J.~Lindstr{\"o}m}, {\em An explicit link between
  gaussian fields and gaussian {M}arkov random fields: the stochastic partial
  differential equation approach}, Journal of the Royal Statistical Society:
  Series B (Statistical Methodology), 73 (2011), pp.~423--498.

\bibitem{meyer2000matrix}
{\sc C.~D. Meyer}, {\em Matrix Analysis and Applied Linear Algebra}, vol.~2,
  SIAM, 2000.

\bibitem{modersitzki2003numerical}
{\sc J.~Modersitzki}, {\em Numerical Methods for Image Registration}, Oxford
  University Press, 2003.

\bibitem{nowak2003efficient}
{\sc W.~Nowak, S.~Tenkleve, and O.~Cirpka}, {\em {Efficient computation of
  linearized cross-covariance and auto-covariance matrices of interdependent
  quantities}}, Mathematical Geology, 35 (2003), pp.~53--66.

\bibitem{OLeary1981}
{\sc D.~P. O'Leary and J.~A. Simmons}, {\em A bidiagonalization-regularization
  procedure for large scale discretizations of ill-posed problems}, SIAM
  Journal on Scientific and Statistical Computing, 2 (1981), pp.~474--489.

\bibitem{PaSa82b}
{\sc C.~C. Paige and M.~A. Saunders}, {\em {{\def\LSQRb{}}Algorithm 583},
  {LSQR:} {S}parse linear equations and least-squares problems}, ACM Trans.
  Math. Soft., 8 (1982), pp.~195--209.

\bibitem{paige1982lsqr}
{\sc C.~C. Paige and M.~A. Saunders}, {\em {LSQR}: An algorithm for sparse
  linear equations and sparse least squares}, ACM Transactions on Mathematical
  Software (TOMS), 8 (1982), pp.~43--71.

\bibitem{park2003super}
{\sc S.~C. Park, M.~K. Park, and M.~G. Kang}, {\em Super-resolution image
  reconstruction: a technical overview}, Signal Processing Magazine, IEEE, 20
  (2003), pp.~21--36.

\bibitem{rasmussen2006gaussian}
{\sc C.~E. Rasmussen and C.~K. Williams}, {\em Gaussian processes for machine
  learning}, the MIT Press, 2 (2006), p.~4.

\bibitem{reichel2012tikhonov}
{\sc L.~Reichel, F.~Sgallari, and Q.~Ye}, {\em {T}ikhonov regularization based
  on generalized {K}rylov subspace methods}, Applied Numerical Mathematics, 62
  (2012), pp.~1215--1228.

\bibitem{renaut2010regularization}
{\sc R.~A. Renaut, I.~Hn{\v{e}}tynkov{\'a}, and J.~Mead}, {\em Regularization
  parameter estimation for large-scale {T}ikhonov regularization using a priori
  information}, Computational Statistics \& Data Analysis, 54 (2010),
  pp.~3430--3445.

\bibitem{renaut2015hybrid}
{\sc R.~A. Renaut, S.~Vatankhah, and V.~E. Ardestani}, {\em Hybrid and
  iteratively reweighted regularization by unbiased predictive risk and
  weighted {GCV}}, arXiv preprint arXiv:1509.00096,  (2015).

\bibitem{saad1980rates}
{\sc Y.~Saad}, {\em On the rates of convergence of the lanczos and the
  block-lanczos methods}, SIAM Journal on Numerical Analysis, 17 (1980),
  pp.~687--706.

\bibitem{saad2003iterative}
{\sc Y.~Saad}, {\em Iterative Methods for Sparse Linear Systems}, SIAM, 2003.

\bibitem{saibaba2012efficient}
{\sc A.~Saibaba and P.~Kitanidis}, {\em Efficient methods for large-scale
  linear inversion using a geostatistical approach}, Water Resources Research,
  48 (2012), p.~W05522.

\bibitem{thompson1976behavior}
{\sc R.~C. Thompson}, {\em The behavior of eigenvalues and singular values
  under perturbations of restricted rank}, Linear Algebra and its Applications,
  13 (1976), pp.~69--78.

\bibitem{van1976generalizing}
{\sc C.~F. Van~Loan}, {\em Generalizing the singular value decomposition}, SIAM
  Journal on Numerical Analysis, 13 (1976), pp.~76--83.

\bibitem{Vogel2002}
{\sc C.~Vogel}, {\em Computational Methods for Inverse Problems}, SIAM,
  Philadelphia, 2002.

\bibitem{wilkinson1965algebraic}
{\sc J.~H. Wilkinson}, {\em The Algebraic Eigenvalue Problem}, vol.~87, Oxford
  Univ Press, 1965.

\end{thebibliography}

\bibliographystyle{siamplain}

\end{document}